\documentclass[10pt]{amsart}
\usepackage{setspace} 
\usepackage{fullpage} 
\usepackage{lmodern} 
\usepackage{graphicx} 
\usepackage{color} 
\usepackage{float} 
\usepackage{accents} 
\usepackage{enumitem} 
\numberwithin{equation}{section}

\usepackage[numbers]{natbib}

\usepackage{amsmath, amsfonts, amssymb, amsthm} 
\usepackage{mathtools} 
\usepackage{xfrac} 
\usepackage{faktor} 
\usepackage{bbm} 
\usepackage{mathrsfs} 
\usepackage[mathscr]{euscript} 
\usepackage{upgreek} 
\usepackage{bm}
\newcommand{\R}{\mathbb{R}} 
\newcommand{\C}{\mathbb{C}} 
\newcommand{\E}{\mathbb{E}} 
\newcommand{\bigO}{\mathrm{O}} 
\newcommand{\littleo}{\mathrm{o}} 

\newcommand{\N}{\mathbb{N}} 
\newcommand{\B}{\mathcal{B}} 
\newcommand{\M}{\mathcal{M}} 
\newcommand{\Leb}{\mathrm{L}} 
\newcommand{\A}{\mathcal{A}} 
\newcommand{\J}{\mathcal{J}} 
\newcommand{\Leg}{\mathcal{L}} 
\newcommand{\Log}{\operatorname{Log}}
\newcommand{\supp}{\operatorname{supp}}
\renewcommand{\Re}{\operatorname{Re}}
\renewcommand{\Im}{\operatorname{Im}}

\newcommand{\mc}[1]{\mathcal{#1}}
\renewcommand{\sf}[1]{\mathsf{#1}}

\newcommand{\scr}[1]{\mathscr{#1}}

\renewcommand{\tilde}[1]{\widetilde{#1}}
\renewcommand{\epsilon}{\varepsilon}
\renewcommand{\geq}{\geqslant}
\renewcommand{\leq}{\leqslant}
\renewcommand{\bar}[1]{\overline{#1}}

\renewcommand{\liminf}{\varliminf}
\renewcommand{\limsup}{\varlimsup}

\newcommand{\ap}{\scr{A}\scr{P}}

\theoremstyle{plain}
\newtheorem{proposition}{Proposition}[section]
\newtheorem{lemma}{Lemma}[section]
\newtheorem{corollary}{Corollary}[section]
\newtheorem{theorem}{Theorem}[section]

\newtheorem*{conjecture}{Conjecture}
\theoremstyle{definition}
\newtheorem{definition}{Definition}[section]

\newtheorem{example}{Example}[section]
\theoremstyle{remark}
\newtheorem{remark}{Remark}[section]

\usepackage{hyperref}
\usepackage{cleveref}

\setlist[enumerate,1]{label=(\arabic*),ref=(\arabic*)}
\setlist[enumerate,2]{label=(\alph*),ref=(\arabic{enumi})(\alph*)}
\setlist[enumerate,3]{label=(\roman*),ref=(\arabic{enumi})(\alph{enumii})(\roman*)}
\setlist[enumerate,4]{label=(\Alph*),ref=(\arabic{enumi}-\alph{enumii}-\roman{enumiii}-\Alph*)} 

\begin{document}

\title{The log-L\'evy moment problem via Berg-Urbanik semigroups}

\author{P.~Patie}
\thanks{This work was partially supported by NSF Grant DMS-1406599, a CNRS grant, and the ARC IAPAS, a fund of the Communaut\'ee fran\c{c}aise de Belgique. Both authors are grateful for the hospitality of the  Laboratoire de Math\'ematiques et de leurs applications de Pau, where this work was initiated.}
\address{School of Operations Research and Information Engineering, Cornell University, Ithaca, NY 14853.}
\email{pp396@orie.cornell.edu}

\author{A.~Vaidyanathan}
\address{Center for Applied Mathematics, Cornell University, Ithaca, NY, 14853.}
\email{av395@cornell.edu}


\begin{abstract}
We consider the Stieltjes moment problem for the Berg-Urbanik semigroups which form a class of multiplicative convolution semigroups on $\R_+$ that is in bijection with the set of Bernstein functions. In \cite{berg:2004}, Berg and Dur\'an  proved that the law of such semigroups is moment determinate (at least) up to time $t=2$, and, for the Bernstein function $\phi(u)=u$, Berg \cite{berg:2005} made the striking observation that  for time $t>2$ the law of this  semigroup is moment indeterminate. We extend these works by estimating the threshold time $\scr{T}_\phi \in [2,\infty]$ that it takes for the law of such Berg-Urbanik semigroups to transition from moment determinacy to moment indeterminacy in terms of simple properties of the underlying Bernstein function $\phi$, such as its Blumenthal-Getoor index. One of the several strategies we implement to deal with the different cases relies on a non-classical Abelian type criterion for the moment problem, recently proved by the authors in \cite{patie:2018a}. To implement this approach we provide detailed information regarding distributional properties of the semigroup such as existence and smoothness of a density, and, the large asymptotic behavior for all $t > 0$ of this density along with its successive derivatives. In particular, these results, which  are original in the L\'evy processes literature, may be of independent interest.
\end{abstract}

\maketitle

\section{Introduction}
The aim of this paper is to study the Stieltjes moment determinacy for multiplicative convolution semigroups $(\nu_t)_{t \geq 0}$, that is semigroups satisfying, for $n, t \geq 0$,
\begin{equation*}
\int_0^\infty x^n \nu_t(dx) = \int_{-\infty}^\infty e^{ny} \hspace{1pt} \mathbb{P}(Y_t \in dy) = e^{t\Psi(n)}
\end{equation*}
where $(Y_t)_{t \geq 0}$ is a one-dimensional L\'evy process such that $\mathbb{E}[e^{nY_t}] < \infty$, for all $n, t \geq 0$. In other words, we study the moment determinacy of the law of a process whose logarithm is a L\'evy process, and we call this problem the log-L\'evy moment problem, for short.


We first point out that if $\Psi(n) = \frac{1}{2}n^2$ then $(\nu_t)_{t \geq 0}$ boils down to the semigroup of the geometric  Brownian motion, whose law  is indeterminate by its moments for all $t > 0$. This is because for any $t > 0$ the geometric Brownian motion is log-normally distributed, and it is well-known that a log-normal distribution is indeterminate by its moments. More generally, in \cite[Theorem 2.1]{patie:2018a} it is proved that the log-L\'evy moment problem is indeterminate for all $t>0$ whenever the associated L\'evy process has a Gaussian component, a case that we exclude from our analysis.

Moreover, Urbanik, in \cite{urbanik:1992}, introduced the multiplicative convolution semigroup of probability densities $(e_t)_{t \geq 0}$ satisfying, for $n, t \geq 0$,
\begin{equation}
\label{eq:classical-Urbanik}
\int_0^\infty x^n e_t(x) dx = (n!)^t = \exp\left({t \sum_{k=1}^n \log k}\right) = \exp\left(t\int_0^\infty (e^{-ny}-1-n(e^{-y}-1)) \frac{dy}{y(e^y-1)}\right),
\end{equation}
and Berg \cite[Theorem 2.5]{berg:2005} discovered that the measure $e_t(x)dx$ is moment determinate if and only if $t \leq 2$. This interesting fact reveals that the log-L\'evy moment problem can be non-trivial, since there can exist a threshold time $\scr{T} \in [0,\infty]$ such that $\nu_t$ is moment determinate for $0 \leq t \leq \scr{T}$ and moment indeterminate for $t > \scr{T}$.

In the same paper, Berg defined a family of multiplicative convolution semigroups $(\nu_t)_{t \geq 0}$ that are in bijection with the set of Bernstein functions $\B$, see \eqref{eq:def-berns} below for definition. In particular, for any $\phi \in \B$, the moments of $\nu_t$  are given, for $n , t \geq 0$, by
\begin{equation}
\label{eq:moments-nu-t}
\M_{\nu_t}(n) = \int_0^\infty x^{n} \nu_t(dx)=\left(\prod_{k=1}^n \phi(k) \right)^t
\end{equation}
where $\M_{\nu_t}$ is called the moment transform of $\nu_t$ and for $n=0$ the product is assumed to be $1$. We call these the Berg-Urbanik semigroups, since \eqref{eq:classical-Urbanik} corresponds to the specific case $\phi(u) = u$ of \eqref{eq:moments-nu-t}. Note that, for a probability measure $\lambda$, there is the notion of Urbanik decomposability semigroups $\mathbb{D}(\lambda)$, which have also been referred to as Urbanik semigroups in the literature, see e.g.~\cite{jurek:1993,jurek:2009}, and are distinct from the semigroups $(\nu_t)_{t \geq 0}$ defined via \eqref{eq:moments-nu-t}. Furthermore in \cite[Theorem 2.2]{berg:2007} it was also shown that, $\M_{\nu_t}$ admits an analytical extension to the right-half plane, and, for $\Re(z) \geq 0$ and $t \geq 0$,
\begin{equation*}
\M_{\nu_t}(z)  = e^{t\Psi(z)}
\end{equation*}
where
\begin{equation}
\label{eq:Psi-representation}
\Psi(z) = z\log \phi(1) + \int_0^\infty (e^{-zy}-1-z(e^{-y}-1)) \frac{\kappa(dy)}{y(e^y-1)}, \quad \text{and} \quad \int_0^\infty e^{-uy}\kappa(dy) = \frac{\phi'(u)}{\phi(u)},
\end{equation}
with $\kappa(dy) = \int_0^y U(dy - r) (r\mu(dr) + \delta_{\sf{d}}(dr))$, where $U$ is the potential measure, $\mu$  the L\'evy measure and $\sf{d}$ the drift of $\phi$, see \eqref{eq:def-pm} and \eqref{eq:Bernstein-function} below for definitions. This is the general form of the right-most equality in \eqref{eq:classical-Urbanik}, and we note that Hirsch and Yor have also derived \eqref{eq:Psi-representation} using different means, see~\cite[Theorem 3.1]{hirsch:2013}. We mention that Hirsch and Yor also offer a nice exposition on the wealth of results by Urbanik in \cite{urbanik:1995}, which continues the investigations started in \cite{urbanik:1992}.

The log-L\'evy moment problem for general Berg-Urbanik semigroups is only partially understood. It is known that any Berg-Urbanik semigroup is moment determinate for $t \leq 2$, see~\cite{berg:2005}, and that there are Berg-Urbanik semigroups that are moment determinate for all $t \geq 0$, see~\cite{berg:2007},  however much less is known concerning moment indeterminacy. We were inspired by Berg's results, in particular his remarkable discovery of the threshold for the classical Urbanik semigroup $(e_t)_{t \geq 0}$, to further study the log-L\'evy moment problem in this setting. In particular, our aim was to understand how to estimate the threshold time $\scr{T}$ from simple properties of the underlying Bernstein function, and our main contribution in this regard is \Cref{thm:threshold-result} below, which provides several new and original results in this area.

One of our approaches stems on a recent Abelian type criterion for the moment problem, established by the authors, that gives a necessary and sufficient condition for moment indeterminacy, see \cite[Theorem 1.2]{patie:2018a}. To utilize this criterion we resort to proving the existence of densities for certain Berg-Urbanik semigroups and study their large asymptotic behavior. To obtain such asymptotics we apply, in a novel and non-standard way, a closure result for Gaussian tails obtained by Balkema et al.~\cite{balkema:1993} combined with some recent Gaussian tail asymptotics estimates due to Patie and Savov \cite{patie:2015}. 

The remaining part of the paper is organized as follows. In \Cref{sec:main-results} we state our main result for the log-L\'evy moment problem, as well some auxiliary results on Berg-Urbanik semigroups and L\'evy processes. In \Cref{sec:examples} we discuss some illustrative examples of Berg-Urbanik semigroups. Finally, \Cref{sec:main-proofs} is devoted to the proofs of the results stated in \Cref{sec:main-results}.

\section{Main results} \label{sec:main-results}

We start with some preliminaries. Let $\phi:[0,\infty) \to [0,\infty)$ be the function defined by
\begin{equation}
\label{eq:Bernstein-function}
\phi(u) = \sf{k} + \sf{d}u + \int_0^\infty (1-e^{-uy})\mu(dy),
\end{equation}
where $\sf{k},\sf{d} \geq 0$ and $\mu$ is a Radon measure on $(0,\infty)$ that satisfies $\int_0^\infty (1 \wedge y) \mu(dy) < \infty$. We write $\B$ for the set of Bernstein functions, which is defined as
\begin{equation} 
\label{eq:def-berns}
\B = \left\{\phi:[0,\infty) \to [0,\infty); \:\phi \text{ is of the form } \eqref{eq:Bernstein-function}\right\}.
\end{equation}
Note that $\B$ is a convex cone, i.e.~for $\phi_1,\phi_2 \in \B$ and $c_1,c_2 > 0$ one has $c_1\phi_1 + c_2\phi_2 \in \B$, and also that the triplet $(\sf{k}, \sf{d}, \mu)$ in \eqref{eq:Bernstein-function} uniquely determines any $\phi \in \B$. We recall that the mapping $u \mapsto \phi'(u)$ is completely monotone, i.e.~$\phi' \in \mathtt{C}^{\infty}(\R^+)$, the space of infinitely continuously differentiable functions on $\R^+$ and  for all $n\in \N $ and $u\geq 0$, $(-1)^{n}\phi^{(n+1)}(u)\geq 0$. It is  well-known that the mapping $u\mapsto \frac{1}{\phi(u)}$ is also completely monotone and the corresponding Radon measure $U$ is the so-called potential measure of (the subordinator associated  to) $\phi$, i.e.~for any $u\geq 0$,
\begin{equation}
\label{eq:def-pm}
\int_{0}^{\infty}e^{-uy}U(dy)=\frac{1}{\phi(u)}.
\end{equation}
We refer to the excellent monograph \cite{schilling:2012} for further information on Bernstein functions, and also to \cite[Section 4]{patie:2015} and \cite[Section 3]{patie:2016}, in which several properties of Bernstein functions that are used in the proofs are collected. In what follows we systematically exclude the trivial Bernstein function $\phi \equiv 0$ since this yields the degenerate convolution semigroup of a Dirac mass at 1 for all time.

A family of measures $(\nu_t)_{t \geq 0}$ is said to be a multiplicative convolution semigroup if, for $t, s \geq 0$ we have $\nu_t \diamond \nu_s = \nu_{t+s}$, where $\diamond$ denotes the product convolution on the multiplicative group $(\R_+,\times)$. Next, we define the moment transform of an integrable function $f:\R_+ \to \R$, and of a probability measure $\rho$ supported on $[0,\infty)$, for (at least) $z \in i\R$ as
\begin{equation*}
\M_f(z) = \int_0^\infty x^{z} f(x)dx, \quad \text{and} \quad \M_\rho(z) = \int_0^\infty x^{z}\rho(dx),
\end{equation*}
and observe that the moment transform is simply a shift of the classical Mellin transform. The moments of $\rho$, if they exist, are given, for $n \geq 0$, by
\begin{equation*}
\M_\rho(n) = \int_0^\infty x^n \rho(dx).
\end{equation*}
We say that a measure $\rho$ supported on $[0,\infty)$ is Stieltjes moment determinate, or simply moment determinate for short, if the sequence $(\M_\rho(n))_{n \geq 0}$ uniquely characterizes the measure $\rho$ among all probability measures supported on $[0,\infty)$ and admitting all moments. Otherwise, we say $\rho$ is moment indeterminate. The moment problem for probability measures supported on $[0,\infty)$ has been intensively studied for many years, going back to the original memoir by Stieltjes \cite{stieltjes:1894}. For excellent references on aspects of the Stieltjes (and other) moment problems see the classic texts~\cite{akhiezer:1965} and~\cite{shohat:1943}, as well as the more recent monograph \cite{schmudgen:2017}.

We now state the definition of Berg-Urbanik semigroups, whose validity is justified by \cite[Theorem 1.8]{berg:2005}.

\begin{definition}
\label{def:Urbanik semigroup}
Let $\phi \in \B$. Then the \emph{Berg-Urbanik semigroup associated to $\phi$} is the unique multiplicative convolution semigroup $(\nu_t)_{t \geq 0}$ of probability measures characterized, for any $ t \geq 0$ and $\Re(z)>0$, by
\begin{equation*}
\label{eq:definition of gen. Urbanik semigroup}
\M_{\nu_t}(z) = e^{t\Psi(z)}
\end{equation*}
where $\Psi$ was defined in \eqref{eq:Psi-representation}. Recall  that, for any $n\in \N$ and $t>0$,  $ e^{t\Psi(n)}=\left(\prod_{k=1}^n \phi(k) \right)^t$.
\end{definition}
Occasionally we write $(\nu_t^\phi)_{t \geq 0}$ to emphasize the dependence of the Berg-Urbanik semigroup on the Bernstein function, but will mostly drop this superscript for convenience. In such cases the Bernstein function will be clear from the context.

\subsection{The log-L\'evy moment problem for Berg-Urbanik semigroups} \label{subsec:threshold}

To describe our first main result  we introduce the \emph{threshold index}. For each $\phi \in \B$ we let $\scr{T}_\phi \in [0,\infty]$ be defined by
\begin{equation*}
\mathscr{T}_\phi = \inf\{t > 0; \: \nu_t^\phi \text{ is indeterminate}\} = \sup\{t > 0;\: \nu_t^\phi \text{ is determinate}\},
\end{equation*}
where we utilize the bijection between $\B$ and the set of Berg-Urbanik semigroups, as well as the convention that $\sup \emptyset = 0$. It is justified to call $\scr{T}_\phi$ a threshold index since $(\nu_t)_{t \geq 0}$ is a multiplicative convolution semigroup and according  to \cite[Lemma 2.2 and Remark 2.3]{berg:2004}, a measure $\mu \diamond \sigma$ is moment indeterminate if $\mu$ is indeterminate and $\sigma \neq c\delta_0$, $c > 0$. Since, for any $\phi \in \B$, $\nu_t$ is moment determinate for $t \leq 2$, it follows that $\scr{T}_\phi \geq 2$. In the case when $\scr{T}_\phi = \infty$ we say the Berg-Urbanik semigroup is \emph{completely determinate}, otherwise if $\scr{T}_\phi \in [2,\infty)$ we say the semigroup is \emph{threshold determinate}. We proceed by defining some subsets of $\B$ that will be useful to state our main results. First, let
\begin{equation*}
\B_d = \{\phi \in \B; \: \sf{d} > 0\}
\end{equation*}
denote the set of Bernstein functions with a positive drift. Next, write 
\begin{equation*}
\B_\J = \{\phi \in \mc{B}; \: \mu(dy) = v(y)dy \text{ with } v \text{ non-increasing}\}
\end{equation*}
and note that this is sometimes referred to as the Jurek class of Bernstein functions, due to \cite{jurek:1985}, see also~\cite[Chapter 10]{schilling:2012}. For a Bernstein function $\phi$ we write $\phi(\infty) = \lim_{u \to \infty} \phi(u) \in (0,\infty]$, and define its Blumenthal-Getoor index as
\begin{equation}
\label{eq:def-Blumenthal-Getoor}
\beta_\phi = \inf\left\lbrace \beta \geq 0; \limsup_{u \to \infty} u^{-\beta}\phi(u) < \infty \right\rbrace \in [0,1],
\end{equation}
noting that this definition coincides with the original one in \cite{blumenthal:1961} for driftless subordinators. We also define the lower index of $\phi$
\begin{equation*}
\delta_\phi = \sup\left\lbrace \delta \geq 0; \: \liminf_{u \to \infty} u^{-\delta}\phi(u) > 0 \right\rbrace,
\end{equation*}
which has appeared in the study of shift-Harnack inequalities for subordinate semigroups, see \cite{deng:2015}. From these definitions it is clear that $0 \leq \delta_\phi \leq \beta_\phi \leq 1$, and moreover one can construct an example for which strict inequality is possible, see \cite[Section 6]{blumenthal:1961}. In view of this, we set
\begin{equation*}
\B_{\asymp} = \{\phi \in \B; \: \delta_\phi = \beta_\phi\}.
\end{equation*}
We are now ready to state our main result regarding the log-L\'evy moment problem for Berg-Urbanik semigroups.
\begin{theorem}
\label{thm:threshold-result}
Let $(\nu_t)_{t \geq 0}$ be the Berg-Urbanik semigroup associated to $\phi \in \B$.
\begin{enumerate}
\item \label{item-1:thm:threshold-result} The inequality
\begin{equation*}
\scr{T}_\phi  \geq \frac{2}{\beta_\phi}
\end{equation*}
holds, and if $\beta_\phi > 0$ and $\limsup_{u \to \infty} u^{-\beta_\phi}\phi(u) < \infty$ then $\nu_{\scr{T}_\phi}$ is moment determinate. In particular, if $\phi(\infty) < \infty$ then $\beta_\phi = 0$ and $(\nu_t)_{t \geq 0}$ is completely determinate.
\end{enumerate}
Moreover, the following hold.
\begin{enumerate}
\addtocounter{enumi}{1} \item \label{item-2:thm:threshold-result} If $\phi \in \B_d$ then $\scr{T}_\phi = 2$, and $\nu_2$ is moment determinate.
\item \label{item-3:thm:threshold-result} If $\phi^\mathfrak{t} \in \B_\J$ for all $\mathfrak{t} \in (0,1)$, then
\begin{equation}
\label{eq:inequality-T}
\frac{2}{\beta_\phi} \leq \scr{T}_\phi \leq \frac{2}{\delta_\phi},
\end{equation}
and hence, if additionally $\phi \in \B_{\asymp}$, then
\begin{equation*}
\scr{T}_\phi = \frac{2}{\beta_\phi}.
\end{equation*}
\item \label{item-4:thm:threshold-result} If there exists $\vartheta \in \B$ such that $\frac{\phi}{\vartheta} \in \B$, then $\scr{T}_\phi \leq \scr{T}_\vartheta$. In particular, if $\vartheta^\mathfrak{t} \in \B_\J$ for all $\mathfrak{t} \in (0,1)$, then
\begin{equation*}
\scr{T}_\phi \leq \frac{2}{\delta_\vartheta}.
\end{equation*}
\end{enumerate}
\end{theorem}

\begin{remark}
\label{rem:cb}
Note that all complete Bernstein functions satisfy the property $\phi^\mathfrak{t} \in \B_\J$ for all $\mathfrak{t} \in (0,1)$. Indeed, writing $\mathbb{H} = \{z \in \C; \: \Im z > 0\}$ for the upper half-plane, we recall that a Bernstein function $\phi$ is said to be a complete if its L\'evy measure $\mu$ has a completely monotone density, or equivalently if $\Im \phi(z) \geq 0$ for all $z \in \mathbb{H}$. Such functions are also sometimes called Pick or Nevanlinna functions in the complex analysis literature. If $\phi$ is a complete Bernstein function, then for $\mathfrak{t} \in (0,1)$ and $z \in \mathbb{H}$,
\begin{equation*}
\Im \phi^\mathfrak{t}(z) = \Im e^{\mathfrak{t}(\log |\phi(z)| + i\arg \phi(z))} = e^{\mathfrak{t}\log |\phi(z)|} \Im e^{i\mathfrak{t} \arg \phi(z)} \geq 0,
\end{equation*}
and hence $\phi^\mathfrak{t}$ is a complete Bernstein function, and in particular its L\'evy measure has a non-increasing density. In particular $u \mapsto (u+\mathfrak{m})^\alpha$ is a complete Bernstein function, for any $\mathfrak{m} \geq 0$, $\alpha \in (0,1)$, and thus $u \mapsto (u+\mathfrak{m})^{\alpha\mathfrak{t}}$ is also a complete Bernstein function, for any $\mathfrak{t} \in (0,1)$. We refer to \cite[Chapter 16]{schilling:2012} for abundant examples of complete Bernstein functions and to \cite[Chapter 6]{schilling:2012} for further details on the theory of complete Bernstein functions; see also \cite{fourati:2011} for some interesting mappings related to complete Bernstein functions.
\end{remark}
\begin{remark}
We mention that for \Cref{item-4:thm:threshold-result} Patie and Savov, see \cite[Proposition 4.4]{patie:2015}, have given sufficient conditions for the ratio of Bernstein functions to remain a Bernstein function, see also Proposition \ref{prop:reference-function} below for another set of sufficient conditions.
\end{remark}

This Theorem is proved in \Cref{subsec:threshold-result-proofs} and the proof makes use of several strategies that will be detailed throughout the rest of the paper.
We proceed by offering some remarks regarding our results in relation to what has been proved in the literature.

First, \Cref{thm:threshold-result}\ref{item-1:thm:threshold-result} provides a generalization of the example provided in \cite{berg:2005} for which the threshold function is infinite. Therein, the author considers the Bernstein function $u \mapsto \frac{u}{u+1}$, for which $\lim_{u \to \infty} \frac{u}{u+1} < \infty$ and therefore trivially $\beta_\phi =0$. However, there exist $\phi \in \B$ such that $\phi(\infty) = \infty$ but $\beta_\phi = 0$, for example the function given, for $u \geq 0$ and any $\lambda > 0$, by
\begin{equation*}
\phi(u) = \log\left(1+\frac{u}{\lambda}\right) = \int_0^\infty (1-e^{-ux}) x^{-1}e^{-\lambda x}dx,
\end{equation*}
which we note is a specific instance of \Cref{ex:2} below. This shows that a Berg-Urbanik semigroup may have unbounded support for all $t > 0$, see~\Cref{thm:smoothness-t}\ref{item-1:thm:smoothness-t} below, but is still completely determinate. Furthermore, in \Cref{thm:threshold-result}\ref{item-1:thm:threshold-result} we provide a condition on $\phi$ that ensures that the lower bound in \eqref{eq:inequality-T} is sharp, in the sense that $\nu_{\scr{T}_\phi}$ is moment determinate. It would be interesting to know what situations can occur when this condition is not fulfilled, in particular if it is possible that $\nu_{\scr{T}_\phi}$ is indeterminate.

In \Cref{thm:threshold-result}\ref{item-2:thm:asymptotics-infinity} we provide an exhaustive claim for the case when $\phi \in \B_d$, thereby generalizing Berg's result that the classical Urbanik semigroup $(e_t)_{t \geq 0}$ is moment determinate if and only if $t \leq 2$, which corresponds to the case $\phi(u) = u$. The proof relies on an application of \Cref{thm:threshold-result}\ref{item-4:thm:threshold-result} to yield the matching upper bound, which shows that $\B_\J$ can serve as a reference class for proving more general estimates. We borrow this idea of using reference objects from~\cite[Section 10]{patie:2015} where the concept of reference semigroups was developed in the context of spectral theory of some non-self-adjoint operators. The fact that one can construct $\phi \in \B$ such that $0 \leq \delta_\phi < \beta_\phi < 1$ shows that the inequality in \eqref{eq:inequality-T} may be far from optimal. Nevertheless, when $\phi \in \B_{\asymp}$, \Cref{thm:threshold-result}\ref{item-3:thm:threshold-result} allows one to classify the behavior of $\scr{T}_\phi$ entirely by the analytical exponent $\beta_\phi$. Finally, as was suggested by an anonymous referee, it is worth emphasizing that for any $T \in (2,\infty)$ there exists a Bernstein function $\phi$ whose associated Berg-Urbanik semigroup has threshold index $\scr{T}_\phi = T$, see e.g.~\Cref{ex:1}.

\subsection{A related moment problem on infinitely divisible moment sequences} \label{subsec:lin-conjecture}

Before we proceed with developing results leading to the proof of \Cref{thm:threshold-result}, we briefly discuss a related moment problem, which requires us to introduce the notion of infinitely divisible moment sequences. A Stieltjes moment sequence $(\mathfrak{m}(n))_{n \geq 0}$ is said to be infinitely divisible if, for any $t > 0$, the sequence $(\mathfrak{m}^t(n))_{n \geq 0}$ is again a Stieltjes moment sequence, and this notion goes back to Tyan who introduced and studied infinitely divisible moment sequences in his thesis \cite{tyan:1975}. By definition, for each $t > 0$, there exists a random variable $X_t$ with moments $(\mathfrak{m}^t(n))_{n \geq 0}$ and it is natural to ask how the moment determinacy of $X_t$ (meaning the moment determinacy of its law) relates to the moment determinacy of $X_1^t$, as a function of $t$. This latter random variable $X_1^t$ is the $t^{th}$-power of a random variable with moments $(\mathfrak{m}(n))_{n \geq 0}$, and it is straightforward that $X_1^t$ has moments given by $(\mathfrak{m}(tn))_{n \geq 0}$. From \Cref{thm:Bernstein-Gamma-Mellin-transform} below it follows that, for any $\phi \in \B$, the moment sequence $(\M_{\nu_1}(n))_{n \geq 0}$ is infinitely divisible and hence Berg-Urbanik semigroups provide a natural setting in which to investigate this question. In what follows we let, for $\phi \in \B$, $X_t(\phi)$ denote the stochastic process whose law at time $t > 0$ is given by $\nu_t^\phi$ and write simply $X(\phi) = X_1(\phi)$, suppressing the dependency on $\phi$ when this causes no confusion.

\begin{theorem}
\label{thm:Lin-conjecture}
Let $\phi \in \B$.
\begin{enumerate}
\item \label{item-1:thm:Lin-conjecture} The random variable $X^t$ is moment determinate for $t < \frac{2}{\beta_\phi}$, and if $\beta_\phi > 0$ and $\limsup_{u \to \infty} u^{-\beta_\phi}\phi(u) < \infty$ then $X^\frac{2}{\beta_\phi}$ is moment determinate.
\end{enumerate}
Moreover, the following hold.
\begin{enumerate}
\addtocounter{enumi}{1} \item \label{item-2:thm:Lin-conjecture} If $\phi \in \B_d$ then $X^t$ is moment determinate if and only if $t \leq 2$.
\item \label{item-3:thm:Lin-conjecture} If $\phi \in \B_{\mc{J}}$ then $X^t$ is moment indeterminate for $t > \frac{2}{\delta_\phi}$. If in addition $\delta_\phi = \beta_\phi$ and~$\limsup_{u \to \infty} u^{-\beta_\phi}\phi(u) < \infty$ then $X^t$ is moment indeterminate if and only if $t > \frac{2}{\beta_\phi}$.
\item \label{item-4:thm:Lin-conjecture} If there exists $\vartheta \in \B$ such that $\frac{\phi}{\vartheta} \in \B$ then, for any $t$ such that $X^t(\vartheta)$ is moment indeterminate, the variable $X^t(\phi)$ is also moment indeterminate.
\end{enumerate}
\end{theorem}

This Theorem is proved in \Cref{subsec:lin-conjecture-proofs}. While \Cref{thm:threshold-result} concerns the $t$-dependent moment determinacy of the process $(X_t)_{t \geq  0}$, \Cref{thm:Lin-conjecture} is the analogous result regarding the moment determinacy of $X^t$, or equivalently of the sequence $(\M_{\nu_1}(tn))_{n \geq 0}$. Note that the conditions in \Cref{thm:Lin-conjecture}\ref{item-3:thm:Lin-conjecture} are weaker than those in \Cref{thm:threshold-result}\ref{item-3:thm:threshold-result}, which shows that the log-L\'evy moment problem is the harder of the two moment problems. In \cite{lin:2017} Lin stated the following conjecture regarding the moment determinacy of infinitely divisible moment sequences.

\begin{conjecture}[Conjecture 1 in \cite{lin:2017}]
Let $(X_t)_{t \geq 0}$ be a stochastic process such that $(\E[X_t^n])_{n \geq 0} = (\mathfrak{m}^t(n))_{n \geq 0}$, i.e.~$(\mathfrak{m}(n))_{n\geq 0}$ is an infinitely divisible moment sequence. Then $X_t$ is moment determinate if and only if $X_1^t$ is moment determinate.
\end{conjecture}

As a corollary of \Cref{thm:threshold-result,thm:Lin-conjecture} we get an affirmative answer to Lin's conjecture for a subclass of Berg-Urbanik semigroups.

\begin{corollary}
\label{cor:Lin-conjecture}
Let $\phi \in \B$ and suppose that any of the following conditions are satisfied:
\begin{enumerate}[label=(\roman*)]
\item $\beta_\phi = 0$,
\item \label{item-2:cor:Lin-conjecture} $\phi \in \B_d$,
\item $\phi \in \B_{\asymp}$ with $\phi^\mathfrak{t} \in \B_\J$ for all $\mathfrak{t} \in (0,1)$,  $\beta_\phi > 0$ and $\limsup_{u \to \infty} u^{-\beta_\phi} \phi(u) < \infty$.
\end{enumerate}
Then Lin's conjecture holds.
\end{corollary}

We point out that recently Berg \cite{berg:2018} proved a related conjecture by Lin (Conjecture 2 in \cite{lin:2017}) concerning the moment sequence $(\Gamma(tn+1))_{n \geq 0}$, which among other things confirms Lin's conjecture (Conjecture 1) for this particular example. Note that the moment sequence $(\Gamma(tn+1))_{n \geq 0}$ corresponds to the Bernstein function $\phi(u) = u$, which falls under the assumption \ref{item-2:cor:Lin-conjecture} in \Cref{cor:Lin-conjecture}.

\subsection{A new Mellin transform representation in terms of Bernstein-Gamma functions} \label{subsec:Mellin}

The proof of \Cref{thm:threshold-result} relies on several intermediate results that are of independent interests. The first one is an alternative representation of $\M_{\nu_t}$. For $a \in \R$ we let $\C_{(a,\infty)} = \{z \in \C; \: \Re(z) > a \}$ and then write $\A_{(a,\infty)}$ for the set of analytic functions on $\C_{(a,\infty)}$. Recall that a function $f:i\R \to \C$ is said to be positive-definite if, for any $s_1,\ldots,s_n \in i\R$ and $z_1,\ldots,z_n \in \C$, $\sum_{i,j=1}^n f(s_i-s_j)z_i\bar{z}_j \geq 0$.

Next, for any $\phi \in \B$ we let $W_\phi:\C_{(0,\infty)} \to \C$ denote the so-called \emph{Bernstein-Gamma function} associated to $\phi$, which is given by
\begin{equation}
\label{eq:generalized-Weierstrass-product}
W_\phi(z) = \frac{e^{-\gamma_\phi z}}{\phi(z)} \prod_{k=1}^\infty \frac{\phi(k)}{\phi(k+z)} e^{\frac{\phi'(k)}{\phi(k)}z}
\end{equation}
where the infinite product is absolutely convergent on at least $\C_{(0,\infty)}$, and
\begin{equation*}
\gamma_\phi = \lim_{n \to \infty} \left( \sum_{k=1}^n \frac{\phi'(k)}{\phi(k)} - \log \phi(n)\right) \in \left[ - \log \phi(1), \frac{\phi'(1)}{\phi(1)}-\log\phi(1) \right].
\end{equation*}
This function, as defined in \eqref{eq:generalized-Weierstrass-product} on $\R_+$ was introduced and studied by Webster \cite{webster:1997}, and was extended (at least) to  $\C_{(0,\infty)}$ by Patie and Savov who introduced the terminology and  studied their analytical properties, such as uniform decay along imaginary lines, in the works \cite[Chapter 6]{patie:2015} and \cite{patie:2016}. The product in \eqref{eq:generalized-Weierstrass-product} can be thought of as a generalized Weierstrass product, as it generalizes the classical Weierstrass product representation for the  gamma function. Indeed, this case can be recovered by setting $\phi(u) = u$, in which case $\gamma_\phi$ boils down to the Euler-Mascheroni constant. Furthermore, $W_\phi$ is the unique positive-definite function that solves the functional equation
\begin{equation*}
W_\phi(z+1) = \phi(z)W_\phi(z), \quad W_\phi(1) = 1,
\end{equation*}
valid for at least $z \in \C_{(0,\infty)}$, see~\cite[Theorem 6.1(3)]{patie:2015}. Write $\Log$ for the branch of the complex logarithm that is analytic on the slit plane $\C \setminus (-\infty,0]$ and satisfies $\Log 1 = 0$, commonly referred to as the principal branch. We use it to define, for $t > 0$ and $z \in \C_{(0,\infty)}$,
\begin{equation*}
W_\phi^t(z) = e^{t\Log W_\phi(z)},
\end{equation*}
as well as $\phi^t(z) = e^{t\Log \phi(z)}$.

\begin{theorem}
\label{thm:Bernstein-Gamma-Mellin-transform}
Let $\phi \in \B$ and let $(\nu_t)_{t \geq 0}$ be the corresponding Berg-Urbanik semigroup. Then, for $t > 0$,
\begin{equation}
\label{eq:Mellin transform equation}
\M_{\nu_t}(z) = \int_0^\infty x^{z} \nu_t(dx) = W_\phi^t(z+1), \quad \Re(z) > -1,
\end{equation}
where $W_\phi:\C_{(0,\infty)} \to \C$ is the Bernstein-Gamma function associated to $\phi$. Moreover, $W_\phi^t \in \A_{(0,\infty)}$ and $W_\phi$ is the unique positive-definite function that solves, for all $t > 0$, the functional equation,
\begin{equation}
\label{eq:functional equation of W_phi^t}
W_\phi^t(z+1) = \phi^t(z) W_\phi^t(z), \quad W_\phi^t(1) =1,
\end{equation}
valid for $z \in \C_{(0,\infty)}$.
\end{theorem}
\begin{remark}
Note that when $t=1$, the equation \eqref{eq:functional equation of W_phi^t} restricted to $\R_+$ was studied by Webster in \cite{webster:1997}, who showed that $\left.W_\phi\right|_{\R_+}$ is the unique log-convex solution to the restricted functional equation.
\end{remark}
\begin{remark}
We point out  that in \cite[Theorem 4.1]{patie:2016} the authors proved that $W_\phi\in \A_{(\mathfrak{d}_\phi,\infty)}$, where
\begin{equation*}
\mathfrak{d}_\phi = \sup\{u \leq 0; \phi(u)=-\infty \ \text{ or } \ \phi(u)=0\} \in [-\infty,0],
\end{equation*}
which is more than what we claim in \Cref{thm:Bernstein-Gamma-Mellin-transform} for $t=1$. However, for $t \neq 1$, $W_\phi^t$ is only defined on the slit plane $\C \setminus (-\infty,0]$ and hence it is not possible to extend the strip of analyticity of $W_\phi^t$ beyond $\C_{(0,\infty)}$.
\end{remark}

This Theorem is proved in \Cref{subsec:Bernstein-Gamma-Mellin-transform-proof} .  Our proof of \eqref{eq:Mellin transform equation} in \Cref{thm:Bernstein-Gamma-Mellin-transform} generalizes an argument given by Berg \cite{berg:2005} for the case $W_\phi(z) = \Gamma(z)$, i.e.~$\phi(u) = u$, which uses the (classical) Weierstrass product representation for the  gamma function. We are able to readily adapt his argument to the generalized Weierstrass product for $W_\phi$ given by \eqref{eq:generalized-Weierstrass-product}, which emphasizes the utility of such a product representation. 

\subsection{Existence, smoothness, and Mellin-Barnes representation of densities} \label{subsec:smoothness}

In this section we obtain the existence of densities for subclasses of Berg-Urbanik semigroups, and quantify their regularities  based on properties of the associated Bernstein function. We write $\mathtt{C}_0(\R_+)$ for the set of continuous functions on $\R_+$ whose limit at infinity is zero. Then, for each $n \in \N$, we write $\mathtt{C}_0^n(\R_+)$ for the set of $n$-times differentiable functions all of whose derivatives belong to $\mathtt{C}_0(\R_+)$, and $\mathtt{C}_0^\infty(\R_+)$ for the set of infinitely differentiable functions all of whose derivatives belong to $\mathtt{C}_0(\R_+)$. Finally, for notational convenience, we write $\mu \in \mathtt{C}_0^n(\R_+)$ to denote that a measure $\mu$ on $\R_+$ has a density, with respect to Lebesgue measure on $\R_+$, and that this density belongs to $\mathtt{C}_0^n(\R_+)$.

To state our next result we need to consider some further subsets of $\B$. Following \cite{patie:2016}, we say that a L\'evy measure $\mu$ satisfies ${\textit{ Condition-}j}$ if $\mu(dy) = v(y)dy$ with $v(0^+)=\infty$, such that $v = v_1 + v_2$ for $v_1,v_2 \in \Leb^1(\R_+)$, and $v_1 \geq 0$ is non-increasing, while $\int_0^\infty v_2(y) dy \geq 0$ satisfies $|v_2(y)| \leq \left(\int_y^\infty v_1(r)dr\right) \vee C$, for some $C > 0$. Given this, we let
 \begin{equation*}
\B_j = \{\phi \in \B; \: \mu \text{ satisfies Condition-}j \}
\end{equation*}
and note that $\B_\J \subset \B_{j}$.

Write $||v||_\infty = \sup_{y \geq 0} |v(y)|$ for the sup-norm of a function on $\R_+$, and set
\begin{equation*}
\B_v = \{\phi \in \B \setminus \B_d ; \: \mu(y) = v(y)dy \text{ with } ||v||_\infty < \infty \},
\end{equation*}
so that $\phi \in \B_v$ implies that $\phi(\infty) < \infty$. We define the quantity $\mathtt{N}_\phi$ as
\begin{equation*}
\mathtt{N}_\phi =
\begin{cases}
\frac{v(0^+)}{\phi(\infty)} & \text{if } \phi \in \B_v, \\
\infty & \text{if } \phi \in \B_j \cup \B_d,
\end{cases}
\end{equation*}
and set
\begin{equation*}
\B_\texttt{N} = \{\phi \in \B_j \cup \B_v \cup \B_d; \: \mathtt{N}_\phi > 0\}.
\end{equation*}
Next, let
\begin{equation*}
\B_\Theta = \left\lbrace\phi \in \B; \: \Theta_\phi = \liminf_{b \to \infty} \frac{1}{|b|}\int_0^{b} \arg \phi(1+iu)du > 0\right\rbrace,
\end{equation*}
and note that $\Theta_\phi \in [0,\frac{\pi}{2}]$ due to \cite[Theorem 3.2(1)]{patie:2016}. In fact, if $\phi \in \B_d$ then $\Theta_\phi = \frac{\pi}{2}$, while if $\lim_{u \to \infty} \phi(u)u^{-\alpha} = C_\alpha$, for $\alpha \in (0,1)$ and a constant $C_\alpha \in (0,\infty)$, then $\Theta_\phi = \alpha\frac{\pi}{2}$ (see \cite[Theorem 3.3]{patie:2016}). Furthermore, there is nothing special about the 1 in $\arg\phi(1+iu)$ as it can be replaced by any $a > 0$ without changing the value of $\Theta_\phi$, which follows from a combination of \cite[Proposition 6.12]{patie:2015} and \cite[Theorem 3.1(1)]{patie:2016}; in the definition of $\B_\Theta$ we simply choose to evaluate $\arg \phi$ along the imaginary line $\Re(z) = 1$ for convenience. For $\theta \in (0,\pi]$ let
\begin{equation*}
\A(\theta) = \{f:\C \to \C; f \text{ is analytic on the sector } |\arg z| < \theta\},
\end{equation*}
that is $\A(\pi)$ denotes the set of functions that are analytic on the slit plane $\C \setminus (-\infty,0]$. Finally, we denote by $\supp(\mu)$ the support of a measure $\mu$.

\begin{theorem}
\label{thm:smoothness-t}
Let $(\nu_t)_{t \geq 0}$ be the Berg-Urbanik semigroup associated to $\phi \in \B$.
\begin{enumerate}
\item \label{item-1:thm:smoothness-t} Assume that $\phi \not\equiv \sf{k}$ for $\sf{k} \geq 0$. If $\phi(\infty) < \infty$ then $\supp(\nu_t) = [0,\phi(\infty)^t]$, otherwise $\supp(\nu_t) = [0,\infty)$ for all $t > 0$.
\item \label{item-2:thm:smoothness-t} If $\phi \in \B_\mathtt{N}$ then, for any $t > \frac{1}{\mathtt{N}_\phi}$, $\nu_t \in \mathtt{C}_0^{\mathfrak{n}(t)}(\R_+)$, i.e.~$\nu_t(dx) = \nu_t(x)dx, x > 0$, where $\mathfrak{n}(t) = \lfloor \mathtt{N_\phi} t \rfloor - 1 \geq 0$. Furthermore, for each $n \leq \mathfrak{n}(t)$, the density $\nu_t(x)$, and its successive derivatives, admit the Mellin-Barnes representation
\begin{equation*}
\nu_t^{(n)}(x) = \frac{(-1)^n}{2\pi i} \int_{c - i\infty}^{c+i\infty} x^{-z-n} \frac{\Gamma(z+n)}{\Gamma(z)}W_\phi^t(z)dz,
\end{equation*}
for any $c,x > 0$.
\item \label{item-3:thm:smoothness-t} If $\phi \in \B_\Theta$, then, for any $0 < t < \frac{\pi}{\Theta_\phi}$, $\nu_t \in \A(\Theta_\phi t)$, and for any $t \geq \frac{\pi}{\Theta_\phi}$, $\nu_t \in \A(\pi)$.
\end{enumerate}
\end{theorem}
\begin{remark}
From \Cref{thm:smoothness-t}\ref{item-1:thm:smoothness-t} it follows that the support of $\nu_t$ is bounded, pointwise in $t$, if and only if $\phi$ is a bounded function. Note that we exclude the case when $\phi \equiv \sf{k}$ as this corresponds to a Berg-Urbanik semigroup with degenerate support, i.e.~$\supp(\nu_t) = \delta_{\sf{k}^t}$.
\end{remark}
This Theorem is proved in \Cref{subsec:smoothness-t-proofs}.   A key ingredient in the proofs of \Cref{thm:smoothness-t}\ref{item-2:thm:smoothness-t} and \Cref{thm:smoothness-t}\ref{item-3:thm:smoothness-t} are estimates for Bernstein-Gamma functions along imaginary lines provided in \cite[Theorem 4.2]{patie:2016}.

The main point of \Cref{thm:smoothness-t}\ref{item-2:thm:smoothness-t} is  to quantify the differentiability of the Berg-Urbanik semigroup as a function of $t$ and simple quantities associated to $\phi$. In this sense our result complements and extends \cite[Theorem 5.2]{patie:2015}, which deals with the differentiability at time $1$. Finally, in \Cref{thm:smoothness-t}\ref{item-3:thm:smoothness-t} we describe the analyticity of $\nu_t$ both as a function of $\phi$ and $t$, and show that the sector of analyticity grows linearly in $t$. This gives rise to another kind of threshold phenomenon, whereby for large enough $t$ we get that the density is analytic on $\C \setminus (-\infty,0]$.

\subsection{Asymptotics at infinity of densities and their successive derivatives} \label{subsec:asymptotics}

In this section we consider a subset of Berg-Urbanik semigroups admitting smooth densities, for all $t > 0$, for which we are able to obtain the exact large asymptotic behavior of the density, as well as for all of its successive derivatives, for all time $t > 0$. We write $f(x) \stackrel{\infty}{\sim} g(x)$ if $\lim_{x \to \infty} \frac{f(x)}{g(x)} = 1$, and $f(x) \stackrel{\infty}{=} \littleo(g(x))$ if $\lim_{x \to \infty} \frac{f(x)}{g(x)} = 0$. The following theorem is the main result of this section, and one of the main results of this paper.

\begin{theorem}
\label{thm:asymptotics-infinity}
Let $\phi \in \B$ be such that $\phi(\infty) =\infty$ with $\phi^\mathfrak{t} \in \B_\J$, for all $\mathfrak{t} \in (0,1)$, and let $(\nu_t)_{t \geq 0}$ be the corresponding Berg-Urbanik semigroup. For any $t > 0$, $\nu_t \in \mathtt{C}_0^\infty(\R_+)$, i.e.~$\nu_t(dx) = \nu_t(x)dx,\: x > 0$, and the densities $\nu_t(x)$ satisfy the following large asymptotic behavior
\begin{equation}
\label{eq:asymptotic-nu-t}
\nu_t\left(x^t\right) \stackrel{\infty}{\sim}  \frac{C_\phi^t}{\sqrt{2\pi t}} \sqrt{x^{1-t}\varphi'(x)} \exp\left({-t\int_\sf{k}^{x} \frac{\varphi(r)}{r}dr}\right)
\end{equation}
where $C_\phi > 0$ is a constant depending only on $\phi$, and $\varphi:[\sf{k},\infty) \to [0,\infty)$ is the continuous inverse of $\phi$. Furthermore, for any $n \in \N$ and $t > 0$, the successive derivatives of the density satisfy
\begin{equation}
\label{eq:asymptotic-nu-t-n}
\nu_t^{(n)}\left(x^t\right) \stackrel{\infty}{\sim} (-1)^n x^{-nt} \varphi^n(x) \nu_t\left(x^t\right)
\end{equation}
which can be specified as follows.
\begin{enumerate}
\item \label{item-1:thm:asymptotics-infinity} If $\phi \in \B_d$ then
\begin{equation*}
\nu_t \left(x^t\right) \stackrel{\infty}{\sim}  \frac{\tilde{C}_\phi^t}{\sqrt{2\pi t}} x^\frac{\sf{d}+t(2\sf{k}-\sf{d})}{2\sf{d}} \exp\left(-\frac{tx}{\sf{d}} +\frac{t}{\sf{d}} \int_\sf{k}^x \frac{E(r)}{r} dr\right)
\end{equation*}
where $\tilde{C}_\phi > 0$ is a constant, and $E(u) \geq 0$ satisfies $E(u) \stackrel{\infty}{=} \littleo(u)$. Furthermore, for any $n \in \N$ and $t > 0$,
\begin{equation*}
\nu_t^{(n)}\left(x^t\right) \stackrel{\infty}{\sim} (-1)^n \sf{d}^n x^{n(1-t)} \nu_t\left(x^t\right).
\end{equation*}
\item \label{item-2:thm:asymptotics-infinity} If $\phi(u) \stackrel{\infty}{\sim} C_\alpha u^\alpha$, for a constant $C_\alpha > 0$ and $\alpha \in (0,1)$, then
\begin{equation*}
\nu_t\left(x^t\right) \stackrel{\infty}{\sim}  \frac{\bar{C}_\phi^t}{\sqrt{2\pi t}} x^{\frac{1-\alpha t}{2\alpha}} \exp\left(-t \alpha C_\alpha^{-\frac{1}{\alpha}} x^{\frac{1}{\alpha}} + t\int_{k}^x \frac{H(r)}{r}dr\right)
\end{equation*}
where $\bar{C}_\phi > 0$ is a constant, and $H(u^\alpha) \stackrel{\infty}{=} \littleo(u)$. Furthermore, for any $n \in \N$ and $t > 0$,
\begin{equation*}
\nu_t^{(n)}\left(x^t\right) \stackrel{\infty}{\sim} (-1)^n C_\alpha^{-\frac{n}{\alpha}} x^{\frac{n}{\alpha}(1-\alpha t)} \nu_t\left(x^t\right).
\end{equation*}
\end{enumerate}
\end{theorem}

\begin{remark}
Note the asymptotic \eqref{eq:asymptotic-nu-t} is a key ingredient in the proof of \Cref{thm:threshold-result} regarding the moment determinacy of the Berg-Urbanik semigroups.
\end{remark}

\begin{remark}
In the special case $\phi(u) = u$ the identity in \eqref{eq:asymptotic-nu-t} boils down to
\begin{equation}
\label{eq:asymptotic-classical-Urbanik}
e_t^{(n)}(x) \stackrel{\infty}{\sim} (-1)^n \frac{(2\pi)^\frac{t-1}{2}}{\sqrt{t}} x^\frac{1-t}{2t} x^{n(\frac{1}{t}-1)} e^{-tx^{\frac{1}{t}}}
\end{equation}
where we recall that $(e_t)_{t > 0}$ stands for the classical Urbanik semigroup, see \eqref{eq:classical-Urbanik}. For $n=0$ and $t>0$ this asymptotic was proved by Berg and L\'opez in \cite{berg:2015}, see also Janson \cite{janson:2010} for an independent proof. In both papers the authors apply a delicate saddle point argument hinging on special properties of the  gamma function such as  the Stirling's formula with Binet remainder for the  gamma function  as in \cite{berg:2015}. Furthermore, Janson outlines how his saddle point argument can be applied to yield the asymptotics in \eqref{eq:asymptotic-classical-Urbanik} for arbitrary $n \in \N$, see~\cite[Remark 6.2]{janson:2010}. It would be interesting to see if a saddle point approach could be applied for general Berg-Urbanik semigroups, using the Mellin transform representation we provide in \Cref{thm:Bernstein-Gamma-Mellin-transform} together with further study of Bernstein-Gamma functions.

\end{remark}
This Theorem is proved in \Cref{subsec:asymptotics-infinity-proofs-1,subsec:asymptotics-infinity-proofs-2}.  There are three main steps in the proof of the asymptotics \eqref{eq:asymptotic-nu-t} and \eqref{eq:asymptotic-nu-t-n}. The first one hinges on a non-classical Tauberian theorem whose version we use is due to Patie and Savov \cite[Proposition 5.26]{patie:2015} but originates from the work of Balkema \cite[Theorem 4.4]{balkema:1995}. It enables us to get the large asymptotic behavior of the densities and of its successive derivatives at time $t = 1$, under the less stringent conditions $\phi \in \B_\J$. Since the conditions to invoke this non-classical Tauberian theorem are difficult to check, one can not follow this path for other times than $1$. Instead, we combine the asymptotic at time $1$ of the densities from \cite[Theorem 5.5]{patie:2015} together with assumption that $\phi^\mathfrak{t} \in \B_\J$, for all $\mathfrak{t} \in (0,1)$, to obtain the asymptotic at time $\mathfrak{t}$. Lastly we adapt to our context a closure result due to Balkema et al.~\cite[Theorem 1.1]{balkema:1993}, which states that the (additive) convolution of probabilities density with Gaussian tails also has a Gaussian tail, to extend the asymptotic from $\mathfrak{t} \in (0,1)$ to all $t > 0$. Our application of this closure result is novel, since we use it not only for the densities (as it is stated in \cite{balkema:1993}) but also for their successive derivatives.

As a by-product of \Cref{thm:asymptotics-infinity} we obtain the large asymptotic behavior of the density and its successive derivatives for the law of certain L\'evy processes, which seems to be new in the L\'evy literature. To state this we briefly recall that a (one-dimensional) L\'evy process $(Y_t)_{t \geq 0}$ is a $\R$-valued stochastic process with stationary and independent increments, that is continuous in probability, and such that $Y_0 = 0$ a.s. For further information regarding L\'evy processes we refer to the monograph \cite{sato:2013}. Note that to each Berg-Urbanik semigroup there exists a corresponding L\'evy process whose characteristic exponent is given by \eqref{eq:Psi-representation}.

\begin{corollary}
Let $\phi \in \B$ be such that $\phi(\infty) =\infty$ with $\phi^\mathfrak{t} \in \B_\J$, for all $\mathfrak{t} \in (0,1)$, and let $(Y_t)_{t \geq 0}$ be a L\'evy process whose characteristic exponent $\Psi$ is given by \eqref{eq:Psi-representation}. Then, for $t > 0$, $\mathbb{P}(Y_t \in dy) = f_t(y)dy, y \in \R$ with $f_t \in \mathtt{C}_0^\infty(\R)$ and, for any $n \geq 0$,
\begin{equation*}
f_t^{(n)}(ty) \stackrel{\infty}{\sim} (-1)^n \frac{C_\phi^t}{\sqrt{2\pi t}} \varphi^n(e^y)\sqrt{e^{(1+t)y}\varphi'(e^y)} \exp\left(-t\int_{\sf{k}}^{e^y} \frac{\varphi(r)}{r}dr\right)
\end{equation*}
where $C_\phi > 0$ is a constant depending only on $\phi$, and $\varphi:[\sf{k},\infty) \to [0,\infty)$ is the continuous inverse of $\phi$.
\end{corollary}

This corollary is obtained by combining \eqref{eq:asymptotic-nu-t} and \eqref{eq:asymptotic-nu-t-n} with the relation $f_t^{(n)}(y) \stackrel{\infty}{\sim} e^{(n+1)y}\nu_t^{(n)}(e^y)$, for any $n \geq 0$, which is established in the proof of \Cref{thm:asymptotics-infinity}. We are not aware of such a detailed description of the large asymptotic behavior for the law of a L\'evy process, for all $t >0$ as well as of its successive derivatives, having appeared in the L\'evy literature before, except in some special cases.

\section{Examples} \label{sec:examples}

In this section we consider two examples of Berg-Urbanik semigroups that illustrate the previous results. 

\begin{example}
\label{ex:1}
Let $\Phi_{\alpha,\mathfrak{a},\mathfrak{b}}$ be the Bernstein function defined, for $u\geq0$, by
\begin{equation*}
\Phi_{\alpha,\mathfrak{a},\mathfrak{b}}(u) = \frac{\Gamma(\alpha u + \mathfrak{a})}{\Gamma(\alpha u + \mathfrak{b})}
\end{equation*}
with $\alpha \in (0,1]$ and $0 \leq \mathfrak{b} <\mathfrak{a} < \mathfrak{b} +1$, where the fact that $\Phi_{\alpha,\mathfrak{a},\mathfrak{b}}$ is a Bernstein functions follows from \cite[Proposition 1 and Remark 1]{kuznetsov:2013}. Next, let, for $\tau \in \R_+$, $G(z|\tau)$ denote the double gamma function, and recall that it satisfies the functional equation
\begin{equation}
\label{eq:double-gamma-functional}
G(z+1|\tau) = \Gamma\left(\frac{z}{\tau}\right) G(z|\tau),
\end{equation}
for $z \in \C_{(0,\infty)}$, with $G(1|\tau) = 1$. We claim that
\begin{equation}
\label{eq:W_phi-equals-double-gamma}
W_{\Phi_{\alpha,\mathfrak{a},\mathfrak{b}}}(z) = C_{\alpha,\mathfrak{a},\mathfrak{b}}\frac{G( z + \frac{\mathfrak{a}}{\alpha}|\frac{1}{\alpha})}{G( z + \frac{\mathfrak{b}}{\alpha}|\frac{1}{\alpha})}, \quad \text{where} \quad C_{\alpha,\mathfrak{a},\mathfrak{b}} = \frac{\Gamma(\mathfrak{b})G(\frac{\mathfrak{b}}{\alpha}|\frac{1}{\alpha})}{\Gamma(\mathfrak{a})G(\frac{\mathfrak{a}}{\alpha}|\frac{1}{\alpha})}.
\end{equation}
Indeed, from \eqref{eq:double-gamma-functional} it follows that
\begin{equation*}
\frac{G( z + 1 + \frac{\mathfrak{a}}{\alpha}|\frac{1}{\alpha})}{G( z + 1 + \frac{\mathfrak{b}}{\alpha}|\frac{1}{\alpha})} = \frac{\Gamma(\alpha z + \mathfrak{a})}{\Gamma(\alpha z + \mathfrak{b})}\frac{G( z + \frac{\mathfrak{a}}{\alpha}|\frac{1}{\alpha})}{G( z + \frac{\mathfrak{b}}{\alpha}|\frac{1}{\alpha})},
\end{equation*}
for $z \in \C_{(0,\infty)}$, and the choice of $C_{\alpha,\mathfrak{a},\mathfrak{b}}$ ensures the required normalization. Hence it remains to prove the uniqueness. To this end we note that, by a Malmsten-type representation for $G(z|\tau)$ due to \cite{lawrie:1994}, we have
\begin{equation}
\label{eq:log-double-gamma}
\log \left(\frac{G( z + \frac{\mathfrak{a}}{\alpha}|\frac{1}{\alpha})}{G( z + \frac{\mathfrak{b}}{\alpha}|\frac{1}{\alpha})} \right) = -c - \kappa z + \int_0^\infty (e^{-zy}-1+zy) f_{\alpha,\mathfrak{a},\mathfrak{b}}(y)dy,
\end{equation}
where $c, \kappa$ are real-constants depending only on the underlying parameters, and
\begin{equation*}
f_{\alpha,\mathfrak{a},\mathfrak{b}}(y) = \frac{(e^{-\frac{\mathfrak{b}}{\alpha}y}-e^{-\frac{\mathfrak{a}}{\alpha} y})}{y(1-e^{-y})(1-e^{-\frac{y}{\alpha}})},
\end{equation*}
see for instance  \cite[(2.15)]{letemplier:2015}. Differentiating the right-hand side of \eqref{eq:log-double-gamma} twice, which is justified by dominated convergence, shows that the ratio of double-gamma functions is log-convex. However, $W_{\Phi_{\alpha,\mathfrak{a},\mathfrak{b}}}$ is the unique log-convex function on $\R_+$ solution to the functional equation, and thus the claim is proved. 

Next, we note that $\Phi_{\alpha,\mathfrak{a},\mathfrak{b}}$ is a complete Bernstein function. Indeed, $\Phi_{\alpha,\mathfrak{a},\mathfrak{b}}$ is obtained by the dilation and translation of the argument of the function $\Phi_{\alpha,\mathfrak{m}}$ below, whose L\'evy measure is easily seen via direct calculation to be completely monotone, and these operations preserve the property of being a complete Bernstein function, which can be seen by using the upper half-plane criterion as outlined in \Cref{rem:cb}. Moreover, the density of the L\'evy measure of $\Phi_{\alpha,\mathfrak{a},\mathfrak{b}}$ is necessarily infinite at 0, which follows from $\Phi_{\alpha,\mathfrak{a},\mathfrak{b}}(\infty) = \infty$, and so $\Phi_{\alpha,\mathfrak{a},\mathfrak{b}} \in \B_j$, which gives by definition that $\mathtt{N}_{\Phi_{\alpha,\mathfrak{a},\mathfrak{b}}} = \infty$. Thus, invoking \Cref{thm:smoothness-t}\ref{item-2:thm:smoothness-t} yields that, for all $t>0$, $\nu_t \in \mathtt{C}_0^\infty(\R_+)$ and since, by Stirling formula, recalled in \eqref{eq:stirling} below, $\Phi_{\alpha,\mathfrak{a},\mathfrak{b}}(u) \stackrel{\infty}{\sim} C u^{\mathfrak{a}-\mathfrak{b}}$, for a constant $C > 0$ and with $\mathfrak{a}-\mathfrak{b} \in (0,1)$, these densities satisfy the large asymptotic behavior specified by~\Cref{thm:asymptotics-infinity}\ref{item-2:thm:asymptotics-infinity}. From \cite[Theorem 3.3(2)]{patie:2016} we get that $\Theta_{\Phi_{\alpha,\mathfrak{a},\mathfrak{b}}} = \frac{(\mathfrak{a}-\mathfrak{b}) \pi}{2}$, see the discussion prior to \Cref{thm:smoothness-t} for the definition, where we may apply this result since $\Phi_{\alpha,\mathfrak{a},\mathfrak{b}} \in \mathcal{B}_\alpha$ with $\ell \equiv 1$ in the notation therein. Hence invoking \Cref{thm:smoothness-t}\ref{item-3:thm:smoothness-t} gives that $\nu_t \in \A(\frac{(\mathfrak{a}-\mathfrak{b}) \pi t}{2})$ for $t < \frac{2}{\mathfrak{a}-\mathfrak{b}}$ and $\nu_t \in \A(\pi)$ for $t \geq \frac{2}{\mathfrak{a}-\mathfrak{b}}$. Finally, the property $\Phi_{\alpha,\mathfrak{a},\mathfrak{b}}(u) \stackrel{\infty}{\sim} C u^{\mathfrak{a}-\mathfrak{b}}$ gives, by \Cref{thm:threshold-result}\ref{item-3:thm:threshold-result}, that $\scr{T}_{\Phi_{\alpha,\mathfrak{a},\mathfrak{b}}} = \frac{2}{\mathfrak{a}-\mathfrak{b}}$ and, by \Cref{thm:threshold-result}\ref{item-1:thm:threshold-result}, we also have that the semigroup is moment determinate at the threshold. As remarked earlier, this example reveals that for any $T \in (2,\infty)$ there exists a Bernstein function, namely $\Phi_{\alpha,\mathfrak{a},\mathfrak{b}}$ with $\mathfrak{a}-\mathfrak{b} = \frac{2}{T}$ and any $\alpha \in (0,1]$, whose associated Berg-Urbanik semigroup has threshold index $\scr{T}_{\Phi_{\alpha,\mathfrak{a},\mathfrak{b}}} = T$.

Now let us now mention that for the special case when $\mathfrak{a} = \alpha \mathfrak{m}+1$ and $\mathfrak{b} = \alpha \mathfrak{m} + 1-\alpha$, where $\mathfrak{m} \in [1-\frac{1}{\alpha},\infty)$, so that $\mathfrak{a}-\mathfrak{b} = \alpha$, some expressions above simplify. Indeed, in this case, the Bernstein function takes the form
\begin{equation*}
\Phi_{\alpha,\mathfrak{m}}(u) = \frac{\Gamma(\alpha u + \alpha \mathfrak{m} + 1)}{\Gamma(\alpha u + \alpha\mathfrak{m} + 1-\alpha)} = \frac{\Gamma(\alpha \mathfrak{m}+1)}{\Gamma(\alpha\mathfrak{m}+1-\alpha)} + \int_0^\infty (1-e^{-uy}) e^{-(\mathfrak{m}+\frac{1}{\alpha})y} (1-e^{-\frac{y}{\alpha}})^{-\alpha - 1} dy,
\end{equation*}
and was studied in the context of the so-called Gauss-Laguerre semigroup in \cite{patie:2017}, see the computations on p.808 therein for the above equality. For $z \in \C_{(0,\infty)}$, the ratio of double gamma functions in \eqref{eq:W_phi-equals-double-gamma} boils down to
\begin{equation*}
W_{\Phi_{\alpha,\mathfrak{m}}}(z) = \frac{\Gamma(\alpha z + \alpha\mathfrak{m}+1-\alpha)}{\Gamma(\alpha\mathfrak{m}+1)},
\end{equation*}
see e.g.~\cite[Lemma 3.1]{patie:2017}, and we also have
\begin{equation*}
\nu_1(x) = \frac{x^{\mathfrak{m}+\frac{1}{\alpha}-1}e^{-x^\frac{1}{\alpha}}}{\Gamma(\alpha\mathfrak{m}+1)}, \ x>0,
\end{equation*}
see~\cite[Equation (3.10)]{patie:2015} and more generally Section 3.3 of the aforementioned paper.
\end{example}

\begin{example}
\label{ex:2}
Let $\phi \in \B$ and consider the  function defined, on $\R^+$, by
\begin{equation*}
\phi_{\ell}(u) = \log\left(\frac{\phi(u+1)}{\phi(1)}\right).
\end{equation*}
Observe that,
\begin{equation*}
\phi_{\ell}'(u) = \log\left(\frac{\phi(u+1)}{\phi(1)}\right)' = \frac{\phi'(u+1)}{\phi(u+1)} = \int_0^\infty e^{-uy} e^{-y}\kappa(dy) = \int_0^\infty e^{-uy} \kappa_e(dy),
\end{equation*}
where we used the last identity in \eqref{eq:Psi-representation} and the last equality serves as a definition for the positive measure $\kappa_e$. It means that $\phi_{\ell}'$ is completely monotone and since $\phi_{\ell}$ is plainly positive on $\R^+$, we deduce that $\phi_\ell \in \B$. Next, as a general result on Bernstein functions gives $\limsup_{u \to \infty} u^{-1}\phi(u) < \infty$, see for instance \cite[Proposition 4.1(3)]{patie:2015}, it follows readily that for any $\beta>0$, $\limsup_{u \to \infty} u^{-\beta}\phi_\ell(u)=0$ and thus $\beta_{\phi_{\ell}} = 0$, see \eqref{eq:def-Blumenthal-Getoor} for definition. Hence, the  Berg-Urbanik semigroup associated to the Bernstein function $\phi_\ell$ is completely determinate. 

As an illustration, we choose, for $\lambda > 0$, $\phi(u)=1+\frac{u}{\lambda} \in \B$ and we have, writing $\phi_\ell=\phi_\lambda$, that
\begin{equation}
\label{eq:phi-lambda}
\phi_\lambda(u) = \log\left(1+\frac{u}{\lambda}\right) = \int_0^\infty (1-e^{-uy}) \frac{e^{-\lambda y}}{y} dy.
\end{equation}
It follows plainly from the right-hand side of the equality \eqref{eq:phi-lambda} that the L\'evy measure of $\phi_\lambda$ is completely monotone, and thus $\phi_\lambda$ is a complete Bernstein function. Furthermore, we have that $\mathtt{N}_{\phi_\lambda} = \infty$, since $\phi_\lambda$ satisfies Condition-$j$ and $\phi_\lambda(\infty) = \infty$. Hence we get from \Cref{thm:smoothness-t}\ref{item-1:thm:smoothness-t} that $\supp(\nu_t) = [0,\infty)$ for all $t > 0$, and from \Cref{thm:smoothness-t}\ref{item-2:thm:smoothness-t} we conclude that for all $t>0$, $\nu_t(dx) = \nu_t(x)dx$ with $\nu_t \in \mathtt{C}_0^\infty(\R_+)$. A straightforward computation yields that the continuous inverse of $\phi_\lambda$ is given by $u \mapsto \lambda(e^u-1)$. Hence, by \Cref{thm:asymptotics-infinity}, we have, for all $t > 0$, that
\begin{equation*}
\nu_t\left(x^t\right) \stackrel{\infty}{\sim} \frac{C^t}{\sqrt{2\pi t}} x^\frac{1-t(1+2\lambda)}{2} \exp\left(-\lambda t \operatorname{Ei}(x) + \frac{x}{2}\right),
\end{equation*}
where $C > 0$ is a constant and $\operatorname{Ei}(x) = -\int_{-x}^\infty \frac{e^{-t}}{t}dt$ is the exponential integral, and we also used the well-known relation $\operatorname{Ei}(x) = \gamma + \log x + \int_0^x \frac{e^r-1}{r}dr$, where $\gamma$ is the Euler-Mascheroni constant.
\end{example}

\section{Proofs of main results} \label{sec:main-proofs}

Throughout the proofs we write $f(x) \stackrel{\infty}{=} \bigO(g(x))$ to denote that $\limsup\limits_{x \to \infty} \left|\frac{f(x)}{g(x)}\right| < \infty$, and recall that $f(x) \stackrel{\infty}{\sim} g(x)$ if $\lim_{x \to \infty} \frac{f(x)}{g(x)} = 1$, and $f(x) \stackrel{\infty}{=} \littleo(g(x))$ if $\lim_{x \to \infty} \frac{f(x)}{g(x)} = 0$. 

\subsection{Proof of \Cref{thm:Bernstein-Gamma-Mellin-transform}}  \label{subsec:Bernstein-Gamma-Mellin-transform-proof} 

We begin with the proof of \eqref{eq:Mellin transform equation} and start by showing that the function $b \mapsto -\Log W_\phi(1+ib)$ is a continuous negative-definite function, i.e.~a continuous function $f$ such that $f(0) \geq 0$ and $u \mapsto e^{-tf(u)}$ is positive-definite for all $t > 0$, see~\cite[Proposition 4.4]{schilling:2012}. As mention in the introduction, this fact has already been established by Berg \cite{berg:2007} and independently by Hirsch and Yor \cite{hirsch:2013} and we shall provide yet another proof utilizing the Weierstrass product representation for $W_\phi$. We follow closely the arguments given by Berg for the proof of \cite[Lemma 2.1]{berg:2005}. First, from \eqref{eq:generalized-Weierstrass-product} we have, for $\Re(z) > 0$,
\begin{equation*}
W_\phi(z) = \frac{e^{-\gamma_\phi z}}{\phi(z)} \prod_{k=1}^\infty \frac{\phi(k)}{\phi(k+z)} e^{\frac{\phi'(k)}{\phi(k)}z},
\end{equation*}
where $\gamma_\phi = \lim_{n \to \infty} \left( \sum_{k=1}^n \frac{\phi'(k)}{\phi(k)} - \log \phi(n)\right) \in \left[ - \log \phi(1), \frac{\phi'(1)}{\phi(1)}-\log\phi(1) \right]$. Hence,
\begin{equation*}
-\Log W_\phi(1+ib) = \gamma_\phi(1+ib) + \Log \phi(1+ib) - \sum_{k=1}^\infty \left( \Log\left(\frac{\phi(k)}{\phi(k+1+ib)}\right) + (1+ib)\frac{\phi'(k)}{\phi(k)}\right).
\end{equation*}
Next, for $n \geq 1$, consider the truncated functions $L_{\phi,n}$ defined by
\begin{align*}
L_{\phi,n}(1+ib)
&= \gamma_\phi(1+ib) + \Log \phi(1+ib) - \sum_{k=1}^n \left( \Log\left(\frac{\phi(k)}{\phi(k+1+ib)}\right) + (1+ib)\frac{\phi'(k)}{\phi(k)}\right) \\
&= L_{\phi,n}(1) + ib\left(\gamma_\phi - \sum_{k=1}^n \frac{\phi'(k)}{\phi(k)}\right) + \sum_{k=1}^{n+1} \Log \frac{\phi(k+ib)}{\phi(k)},
\end{align*}
where
\begin{equation*}
L_{\phi,n}(1) = \gamma_\phi - \sum_{k=1}^n \frac{\phi'(k)}{\phi(k)} + \log\phi(n+1) = \gamma_\phi - g(n),
\end{equation*}
and the last equality serves to define $g(n)$. We claim that $n \mapsto g(n)$ is non-decreasing with $\lim_{n \to \infty} g(n) = \gamma_\phi$. Indeed, we have from \cite[Proposition 4.1(4)]{patie:2015} that $\frac{1}{\phi}$ is completely monotone so that $\frac{\phi'}{\phi}$ is completely monotone, as the product of two completely monotone functions. Thus $u \mapsto \frac{\phi'(u)}{\phi(u)}$ is non-increasing, and we get that
\begin{equation*}
\log \frac{\phi(n+2)}{\phi(n+1)} = \int_{n+1}^{n+2} \frac{\phi'(u)}{\phi(u)}du \leq \frac{\phi'(n+1)}{\phi(n+1)},
\end{equation*}
which yields
\begin{equation*}
g(n+1)-g(n) = \frac{\phi'(n+1)}{\phi(n+1)} - \log \frac{\phi(n+2)}{\phi(n+1)} \geq 0.
\end{equation*}
Additionally, by \cite[Proposition 4.1(6)]{patie:2015}
\begin{equation*}
\lim_{n \to \infty} \frac{\phi(n+1)}{\phi(n)} = 1,
\end{equation*}
so that
\begin{equation*}
\lim_{n \to \infty} g(n) = \lim_{n \to \infty} \left(\sum_{k=1}^n \frac{\phi'(k)}{\phi(k)} - \log\phi(n) + \log\phi(n) - \log\phi(n+1) \right) = \gamma_\phi - \lim_{n \to \infty} \log \frac{\phi(n+1)}{\phi(n)} = \gamma_\phi.
\end{equation*}
Putting all of these observations together, we conclude that $L_{\phi,n}(1) \geq 0$. Furthermore, for any $a \in \R$ the function $b \mapsto iab$ is continuous negative-definite, and for any $1 \leq k \leq n +1$, $b \mapsto \Log \frac{\phi(k+ib)}{\phi(k)}$ is continuous negative-definite since $u \mapsto \log \frac{\phi(k+u)}{\phi(k)}$ is a Bernstein function, as the composition of two Bernstein functions, see \cite[Corollary 3.8(iii)]{schilling:2012}. This shows that $L_{\phi,n}(1+ib)$ is a continuous negative-definite function, and since $\lim_{n \to \infty} L_{\phi,n}(1+ib) = -\Log W_\phi(1+ib)$ pointwise it follows that $b \mapsto -\Log W_\phi(1+ib)$ is a continuous negative-definite function.

Consequently, using the homeomorphism $x \mapsto e^x$ between $\R$ and $(0,\infty)$, we find that there exists a unique multiplicative convolution semigroup $(\mathcal{V}_t)_{t \geq 0}$ such that
\begin{equation}
\label{eq:semigroup-mu-t}
\int_0^\infty y^{ib} \mathcal{V}_t(dy) = W_\phi^t(1+ib).
\end{equation}
From \cite[Theorem 6.1]{patie:2015} we know that $W_\phi \in \A_{(0,\infty)}$ and hence $W_\phi^t \in \A_{(0,\infty)}$ for any $t > 0$. Thus the identity in \eqref{eq:semigroup-mu-t} extends to
\begin{equation*}
\int_0^\infty y^{z-1} \mathcal{V}_t(dy) = W_\phi^t(z),
\end{equation*}
for $z \in \C_{(0,\infty)}$. However, again from \cite[Theorem 6.1]{patie:2015}, we have that $\M_{\nu_1}(z-1) = W_\phi(z)$ and thus $\mathcal{V}_1 = \nu_1$, since the Mellin transform uniquely characterizes a probability measure. By uniqueness of convolution semigroups it then follows that $\mathcal{V}_t = \nu_t$ for all $t \geq 0$, and thus \eqref{eq:Mellin transform equation} is established. Finally, from \cite[Theorem 6.1]{patie:2015} we have that $W_\phi:\C_{(0,\infty)} \to \C$ is the unique positive-definite function, i.e.~the Mellin transform of a probability measure, that satisfies the functional equation
\begin{equation*}
W_\phi(z+1) = \phi(z)W_\phi(z), \quad W_\phi(1) = 1,
\end{equation*}
for $z \in \C_{(0,\infty)}$, from which the last claim follows.

\subsection{Proofs for \Cref{subsec:smoothness}} \label{subsec:smoothness-t-proofs}

\subsubsection{Proof of \Cref{thm:smoothness-t}\ref{item-1:thm:smoothness-t}} 

It is immediate from \cite[Theorem 1.5]{berg:2004} that $\phi(\infty) = \infty$ implies $\supp(\nu_t)$ is unbounded, and we also get from \cite[Theorem 5.2(1)]{patie:2015} that $\supp(\nu_1) = [0,\infty)$. By the homeomorphism $x \mapsto e^x$ between $\R$ and $(0,\infty)$ mentioned above, together with the fact that the boundedness from below of the support of the law of a L\'evy process is time-independent, see \cite[Theorem 24.7]{sato:2013}, we then conclude that $\supp(\nu_t) = [0,\infty)$ for all $t > 0$. Hence, we suppose that $\phi(\infty) \in (0,\infty)$. To prove the claim we will rely on the following auxiliary result: for any measure $\mu$ on $\R_+$, $\supp(\mu) \subseteq [0,c]$, for $c > 0$, if and only if $\int_0^\infty x^n \mu(dx) \stackrel{\infty}{=} \bigO(c^n)$, see~\cite[Lemma 2.9]{berg:2004}. Since for any $\phi \in \B$ we have, by definition, that $\phi'$ is completely monotone it follows that all Bernstein functions are non-decreasing on $\R_+$. Thus we have, for any $n \geq 0$,
\begin{equation*}
\M_{\nu_t}(n) = \left( \prod_{k=1}^n \phi(k)\right)^t \leq \phi(\infty)^{nt}.
\end{equation*}
By the quoted result, the above estimate implies that $\supp(\nu_t) \subseteq [0,\phi(\infty)^t]$. For the reverse inclusion, let $\epsilon > 0$ be small and choose $N_{\epsilon,\phi}$ large enough (depending on $\epsilon$ and $\phi$) such that for $k \geq N_{\epsilon,\phi}-1$ we have $\phi(k) \geq \phi(\infty) - \epsilon > 0$. Then, for $n \geq N_{\epsilon,\phi}$ and again since $\phi$ is non-decreasing,
\begin{equation*}
\M_{\nu_t}(n) = \left( \prod_{k=1}^{N_{\epsilon,\phi}-1} \phi(k)\right)^t \left( \prod_{k=N_{\epsilon,\phi}}^{n} \phi(k)\right)^t \geq C_{\epsilon,\phi,t} (\phi(\infty)-\epsilon)^{nt},
\end{equation*}
where
\begin{equation*}
C_{\epsilon,\phi,t} = \frac{\phi(1)^{(N_{\epsilon,\phi}-1)t}}{(\phi(\infty)-\epsilon)^{N_{\epsilon,\phi}t}}
\end{equation*}
is a constant, which depends only on $\epsilon$, $\phi$, and $t$. Since $\epsilon > 0$ is arbitrary this estimate shows that $\supp(\nu_t)$ cannot be contained in any sub-interval of $[0,\phi(\infty)^t]$. Thus we must either have that $\supp(\nu_t) = [0,\phi(\infty)^t]$ or $\supp(\nu_t) = \delta_{\phi(\infty)^t}$, a Dirac mass at the point $\phi(\infty)^t$. In the latter case,
\begin{equation*}
\M_{\nu_t}(n) = \phi(\infty)^{nt} = \left( \prod_{k=1}^n \phi(k)\right)^t,
\end{equation*}
for all $n \geq 0$ and $t > 0$, from which it follows that $\phi$ must be constant.

\subsubsection{Proof of \Cref{thm:smoothness-t}\ref{item-2:thm:smoothness-t}}

We split the proof into two cases. First, suppose that $\mathtt{N}_\phi = \infty$, which implies that $\phi \in \B_d \cup \B_j$. Then one may invoke \cite[Theorem 4.2(3)]{patie:2016} to get that, for any $p \geq 0$ and $a > 0$,
\begin{equation*}
\lim_{|b| \to \infty} |b|^p |W_\phi(a+ib)| = 0,
\end{equation*}
where $W_\phi:\C_{(0,\infty)} \to \C$ is the Bernstein-Gamma function associated to $\phi$. Hence, for any $q \geq 0$ and $t > 0$ fixed,
\begin{equation*}
\lim_{|b| \to \infty} |b|^q |W_\phi(a+ib)|^t = 0,
\end{equation*}
which yields the estimate
\begin{equation*}
|W_\phi^t(a+ib)| \stackrel{\infty}{=} \bigO(|b|^{-q}),
\end{equation*}
uniformly on bounded $a$-intervals, i.e.~uniformly on bounded intervals of $a \in (0,\infty)$. Indeed, the functions $E_\phi$ and $R_\phi$ in \cite[Theorem 4.2]{patie:2016} are uniformly bounded for all $a > 0$ and all $\phi \in \B$, while the function $G_\phi$ in \cite[Theorem 4.2]{patie:2016} depends only on $a$ and $G_\phi(a) \leq a \log\phi(1+a)$, so that $G_\phi$ is uniformly bounded on bounded $a$-intervals, see also \cite[Remark 4.3]{patie:2016}. By \Cref{thm:Bernstein-Gamma-Mellin-transform} we know that $\M_{\nu_t}(z-1) = W_\phi^t(z)$, for $\Re(z) > 0$, so the estimate for $W_\phi^t$ established above, together with the fact that $W_\phi^t \in \A_{(0,\infty)}$, justifies the use of Mellin inversion, see e.g.~\cite{titchmarsh:1986}, to conclude that, for any $c > 0$,
\begin{equation}
\label{eq:nu-t Mellin representation}
\nu_t(x) = \frac{1}{2\pi i}\int_{c-i\infty}^{c+i\infty} x^{-z} W_\phi^t(z)dz.
\end{equation}
Note that the integrand in \eqref{eq:nu-t Mellin representation} is absolutely integrable for any $x > 0$, since $|x^{-(c+ib)}| |W_\phi^t(c+ib)| = x^{-c} |W_\phi^t(c+ib)|$ and $|W_\phi^t(c+ib)| \stackrel{\infty}{=} \bigO(|b|^{-1})$, for $|b|$ large enough. Taking $\lim_{x \to \infty} \nu_t(x)$ in \eqref{eq:nu-t Mellin representation} and using the dominated convergence theorem to interchange the limit and the integral gives that $\nu_t \in \mathtt{C}_0(\R_+)$. However, since for any $q \geq 0$ and $a > 0$, $|W_\phi^t(a+ib)| \stackrel{\infty}{=} \bigO(|b|^{-q})$, we deduce that, for any $n = 0,1,2,\ldots$, $z \mapsto z^n |W_\phi^t(z)|$ is absolutely integrable and uniformly decaying on a complex strip containing $c+n+i\R$, see e.g.~\cite[Section 1.7.4]{patie:2015}, and thus we get
\begin{equation*}
\nu_t^{(n)}(x) = \frac{(-1)^n}{2\pi i}\int_{c+n-i\infty}^{c+n+i\infty} x^{-z} \frac{\Gamma(z)}{\Gamma(z-n)} W_\phi^t(z-n) dz.
\end{equation*}
By the change of variables $z \mapsto z+n$ then yields the claimed Mellin-Barnes representation,
\begin{equation*}
\nu_t^{(n)}(x) = \frac{(-1)^n}{2\pi i}\int_{c-i\infty}^{c+i\infty} x^{-z-n} \frac{\Gamma(z+n)}{\Gamma(z)} W_\phi^t(z) dz,
\end{equation*}
where we note that the integrand is absolutely integrable by Stirling's formula for the gamma function, see \eqref{eq:stirling} below. Using the dominated convergence theorem once more to evaluate the limit at infinity yields that $\nu_t \in \mathtt{C}_0^\infty(\R_+)$. 

Next, suppose that $\phi \in \B_v$, i.e.~$\mathtt{N}_\phi = \frac{v(0^+)}{\phi(\infty)} \in (0,\infty)$. Another application of \cite[Theorem 4.2]{patie:2016} yields that, for $a > 0$ fixed and any $\epsilon > 0$,
\begin{equation*}
\lim_{|b| \to \infty} |b|^{\mathtt{N}_\phi-\epsilon} |W_\phi(a+ib)| = 0,
\end{equation*}
while
\begin{equation*}
\lim_{|b| \to \infty} |b|^{\mathtt{N}_\phi+\epsilon} |W_\phi(a+ib)| = \infty.
\end{equation*}
The first equality thus guarantees that, for $t > 0$ and any $\epsilon > 0$,
\begin{equation}
\label{eq:Mellin-decay-N_phi}
\lim_{|b| \to \infty} |b|^{\mathtt{N}_\phi t - \epsilon} |W_\phi(a+ib)|^t = 0.
\end{equation}
Now let $t > \frac{1}{\mathtt{N}_\phi}$ and observe that $\mathfrak{n}(t) = \lfloor \mathtt{N}_\phi t \rfloor - 1 \geq 0$ and is the largest integer less than or equal to $\mathtt{N}_\phi t - 1$. Choose $\epsilon$ such that $\mathtt{N}_\phi t - 1 - \mathfrak{n}(t)> \epsilon > 0$. Then, by \eqref{eq:Mellin-decay-N_phi}, it follows that, uniformly on bounded $a$-intervals, and for $|b|$ large enough
\begin{equation*}
|W_\phi(a+ib)|^t \leq C |b|^{-1- \mathfrak{n}(t)-\epsilon},
\end{equation*}
for $C > 0$ a constant. Since the right-hand side is uniformly integrable and $W_\phi^t$ is analytic on $\C_{(0,\infty)}$, another application of the Mellin inversion  formula and dominated convergence allows us to conclude that $\nu_t \in \mathtt{C}_0^{\mathfrak{n}(t)}(\R_+)$. The Mellin-Barnes representation follows as in the previous case.

\subsubsection{Proof of \Cref{thm:smoothness-t}\ref{item-3:thm:smoothness-t}}  \label{subsec:analyticity-t-proofs}

Since $\phi \in \B_{\Theta}$ we have, for any $\epsilon > 0$ and $|b|$ large enough,
\begin{equation}
\label{eq:estimate-A_phi}
A_\phi(a+ib) \geq (\Theta_\phi - \epsilon)|b|,
\end{equation}
where $A_\phi(a+ib) = \int_0^b \arg \phi(a+iu) du$. Invoking \cite[Theorem 4.2(1)]{patie:2016} gives, for any $a > 0$,
\begin{equation*}
|W_\phi(a+ib)|^t = C_{\phi,a,t} \left(\frac{\phi(a)}{|\phi(a+ib)|}\right)^\frac{t}{2} e^{-tA_\phi(a+ib)},
\end{equation*}
where $C_{\phi,a,t} > 0$ is a constant depending only on $\phi$, $a$ and $t$. Since \cite[Proposition 3.1(9)]{patie:2016} gives that $|\phi(a+ib)| \geq \phi(a)$, it follows from the estimate for $A_\phi$ in \eqref{eq:estimate-A_phi} that, for $\epsilon$ small enough such that $\Theta_\phi t - \epsilon > 0$,
\begin{equation}
\label{eq:estimate-W_phi}
|W_\phi(a+ib)|^t  \stackrel{\infty}{=} \bigO\left(e^{-(\Theta_\phi t-\epsilon)|b|}\right),
\end{equation}
where the big-$\bigO$ estimate holds pointwise in $a$, and thus uniformly on bounded $a$-intervals. By similar arguments as given in the proof of \Cref{thm:smoothness-t}\ref{item-2:thm:smoothness-t} above, it follows that $\nu_t \in \mathtt{C}_0^\infty(\R_+)$, and hence we have the Mellin-Barnes representation for $\nu_t$
\begin{equation}
\label{eq:Mellin-Barnes-2}
\nu_t(x) = \frac{1}{2\pi i}\int_{c-i\infty}^{c+i\infty} x^{-z} W_\phi^t(z)dz,
\end{equation}
for any $c > 0$. To show that $\nu_t$ is analytic on the claimed sector it suffices to analytically extend the right-hand side of \eqref{eq:Mellin-Barnes-2}, which amounts to replacing $x$ by a suitable complex number. Let $\epsilon > 0$ be fixed and consider $w \in \C$ such that $|\arg w| < \Theta_\phi t-\epsilon$. From the estimate \eqref{eq:estimate-W_phi} it follows that, for any $c > 0$ and $b \in \R$,
\begin{equation*}
|w^{-(c+ib)} W_\phi^t(c+ib)| \leq e^{|b| |\arg w|} |W_\phi(c+ib)|^t \stackrel{\infty}{=} \bigO\left(e^{-(\Theta_\phi t-\epsilon-|\arg w|)|b|}\right),
\end{equation*}
and by choice of $w$ the right-hand side is integrable in $b$. Thus the integrand on the right-hand side of \eqref{eq:Mellin-Barnes-2} is well-defined for $|\arg w| < \Theta_\phi t-\epsilon$, which by uniqueness of the analytic extension gives that $\nu_t \in \A(\Theta_\phi t-\epsilon)$. Since $\epsilon > 0$ is arbitrary we get $\nu_t \in \A(\Theta_\phi t)$, and thus for $t > \frac{\pi}{\Theta_\phi}$ we have $\nu_t \in \A(\pi)$.

\subsection{Proofs for \Cref{subsec:asymptotics}} \label{subsec:Proofs for asymptotics results}

The proof of \Cref{thm:asymptotics-infinity} combines ideas from several different areas. Hence we first state some definitions, and detail some lemmas and propositions that will be useful in the proof. We say that a function $s:(a,\infty) \to (0,\infty)$, for some $a \geq -\infty$, is \emph{self-neglecting} if
\begin{equation*}
\lim_{u \to \infty} \frac{s(u+ws(u))}{s(u)} = 1, \quad \text{locally uniformly in } w \in \R.
\end{equation*}
Furthermore, we say a function $G:(a,\infty) \to \R$ is \emph{asymptotically parabolic} if it is twice differentiable with $G'' > 0$ on $(a,\infty)$, and if its scale function $s_G(u) = (G''(u))^{-\frac{1}{2}}$ is self-neglecting. Denote the set of asymptotically parabolic functions by $\ap$ and note that it is a convex cone. A function $h:(a,\infty) \to (0,\infty)$ is said to be \emph{flat} with respect to $G$ if
\begin{equation}
\label{eq:def-flat-functions}
\lim_{u \to \infty} \frac{h(u+ws_G(u))}{h(u)} = 1, \quad \text{locally uniformly in } w \in \R,
\end{equation}
where $s_G$ is the scale function of $G$. In the following lemma we collect some properties of flat and asymptotically parabolic functions.

\begin{lemma}
\label{lem:flat-big-O}
Let $G \in \ap$ and $h$ be flat with respect to $G$.
\begin{enumerate}
\item \label{item-1:lem:flat-big-O} The function $u \mapsto 1/h(u)$ is flat with respect to $G$.
\item \label{item-2:lem:flat-big-O} For any $c > 0$, the function $u \mapsto h(cu)$ is flat with respect to $G$.
\item \label{item-3:lem:flat-big-O} The identity function is flat with respect to $G$ and, for any $\alpha > 0$, the function $u \mapsto h^\alpha(u)$ is flat with respect to $G$. In particular, for any $n \geq 0$, the function $u \mapsto u^n$ is flat with respect to $G$.
\item \label{item-4:lem:flat-big-O} The function $h$ satisfies
\begin{equation*}
\lim_{u\to \infty} \frac{\log h(u)}{G(u)} = 0.
\end{equation*}
\item \label{item-5:lem:flat-big-O} For any $c > 0$, the function $u \mapsto cG\left(\frac{u}{c}\right) \in \ap$.
\end{enumerate}
\end{lemma}

\begin{proof}
The first claim is obvious from the definition in \eqref{eq:def-flat-functions}. Let $c > 0$ and consider the function $h_c$ defined by $h_c(u) = h(cu)$. Then, writing $v = cu$,
\begin{equation*}
\lim_{u \to \infty} \frac{h_c(u+ws_G(u))}{h_c(u)} = \lim_{u \to \infty} \frac{h(cu+cws_G(cu))}{h(cu)} = \lim_{v \to \infty} \frac{h(v+cws_G(v))}{h(v)} =1,
\end{equation*}
where the last limit follows from fact that \eqref{eq:def-flat-functions} holds locally uniformly for $w \in \R$. For the third claim, note that $s_G(u) \stackrel{\infty}{=} \littleo(u)$, see e.g.~\cite[Lemma 3.1]{patie:2018a} so that, locally uniformly in $w \in \R$,
\begin{equation*}
\lim_{u \to \infty} \frac{u+ws_G(u)}{u} = 1 + w\lim_{u \to \infty} \frac{s_G(u)}{u} = 1.
\end{equation*}
The fact that, for $\alpha > 0$, $u \mapsto h^\alpha(u)$ is flat follows trivially from the definition, and the proof of the fourth item is essentially known in the literature, see again~\cite[Lemma 3.1]{patie:2018a}. Finally, for the proof of the last claim, write $\tilde{G}(u) = cG_c\left(\frac{u}{c}\right)$ and $s_{\tilde{G}}$ for the corresponding scale function. Then $s_{\tilde{G}}(u) = \sqrt{c}s_G\left(\frac{u}{c}\right)$ so that, for $w \in \R$,
\begin{equation*}
\frac{s_{\tilde{G}}(u+ws_{\tilde{G}}(u))}{s_{\tilde{G}}(u)} = \frac{s_G\left(\frac{u}{c}+ \sqrt{c} w s_G\left(\frac{u}{c}\right)\right)}{s_G\left(\frac{u}{c}\right)}
\end{equation*}
and the self-neglecting property of $s_G$ carries over readily to $s_{\tilde{G}}$.
\end{proof}

In the next lemma we collect some properties about the specific asymptotically parabolic functions that will play a role in the proof of \Cref{thm:asymptotics-infinity}. To state it we recall that the Legendre transform of a convex function $\psi:\R \to \R$, which we denote as $\Leg_\psi$, is given by
\begin{equation*}
\Leg_\psi(y) = \sup_{u \in \R} \{uy - \psi(u)\}.
\end{equation*}
If in addition $\psi \in \mathtt{C}^1(\R)$ then the above supremum is achieved at the unique point $u = \psi'^{-1}(y)$, and hence
\begin{equation*}
\Leg_\psi(y) = y\psi'^{-1}(y) - \psi(\psi'^{-1}(y)).
\end{equation*}
The variables $u$ and $y$ obeying the relations $y = \psi'(u)$ and $u = \psi'^{-1}(y)$ are called conjugate variables.

\begin{lemma}
\label{lem:self-neglecting}
Let $\phi \in \B_\mc{J}$ be such that $\phi(\infty) = \infty$. Then the function $s_G:\R_+ \to \R_+$ defined by
\begin{equation*}
s_G(u) = \sqrt{\frac{\phi(u)}{\phi'(u)}}
\end{equation*}
is self-neglecting, and consequently $G \in \ap$, where $G:(1,\infty) \to \R$ is the function defined by
\begin{equation}
\label{eq:def-G}
G(u) = \int_1^u \log \phi(r) dr + \log \phi(1).
\end{equation}
The Legendre transform of $G$ is given by
\begin{equation*}
\Leg_G(y) = \int_\sf{k}^{e^y} \frac{\varphi(r)}{r} dr - \int_\sf{k}^{\phi(1)} \frac{\varphi(r)}{r} dr
\end{equation*}
where $\varphi:[\sf{k},\infty) \to [0,\infty)$ is the continuous inverse of $\phi$, and $y$ and $u$ are conjugate variables related by $y = \log \phi(u)$ and $u = \varphi(e^y)$. Furthermore, $\Leg_G \in \ap$.
\end{lemma}

\begin{proof}
The fact that $s_G$ is self-neglecting was proved in \cite[Proposition 5.40]{patie:2015} under the additional condition that $\sf{k} = \phi(0) > 0$. However, an inspection of the proof reveals that this property is not crucial for the self-neglecting property of $s_G$. Differentiating $G$ twice shows that $s_G$ is indeed the scale function of $G$, and hence $G \in \ap$.

Taking derivatives in \eqref{eq:def-G} we get $G'(u) = \log \phi(u)$ so that the conjugate variables are $y = \log \phi(u)$ and $u = \varphi(e^y)$. Also, by integration by parts we can rewrite $G$ as
\begin{equation*}
G(u) = u\log \phi(u) - \int_1^u \frac{r\phi'(r)}{\phi(r)}dr.
\end{equation*}
Hence,
\begin{equation*}
\Leg_G(y) = y \varphi(e^y) - G(\varphi(e^y)) = \int_1^{\varphi(e^y)} \frac{r\phi'(r)}{\phi(r)} dr = \int_{\phi(1)}^{e^y} \frac{\varphi(r)}{r} dr = \int_{\sf{k}}^{e^y} \frac{\varphi(r)}{r} dr  - \int_\sf{k}^{\phi(1)} \frac{\varphi(r)}{r} dr
\end{equation*}
where the third equality follows by the change of variables $r=\varphi(w)$. Finally, the fact that $\Leg_G \in \ap$ follows from a closure property of $\ap$ with respect to the Legendre transform, see~\cite[Theorem 5.3]{balkema:1993}.
\end{proof}

In the final lemma before the proof we collect some properties concerning additive convolution, especially a stability property for Gaussian tails under additive convolution. We write $*$ for the additive convolution of suitable functions $f, g : \R \to \R$, that is
\begin{equation*}
(f*g)(x) = \int_{-\infty}^\infty f(x-y)g(y)dy = \int_{-\infty}^\infty f(y)g(x-y)dy,
\end{equation*}
with the additive convolution of measures being defined similarly. A probability density $f$ is said to have a \emph{Gaussian tail} if $f(y) \stackrel{\infty}{\sim} \eta(y)e^{-\psi(y)}$ for some $\psi \in \ap$ and some $\eta$ flat with respect to $\psi$.

\begin{lemma}
\label{lem:tail-asymptotics} $ $
\begin{enumerate}
\item \label{item-1:lem:tail-asymptotics} Let $(\nu_t)_{t \geq 0}$ be a multiplicative convolution semigroup and let, for each $t > 0$, $f_t$ be the pushforward measure under the map $x \mapsto \log x$. Then $(f_t)_{t \geq 0}$ is an additive convolution semigroup, i.e.~for $t,s \geq 0$, $f_t * f_s = f_{t+s}$.
\item \label{item-2:lem:tail-asymptotics} Let $f, g \in \Leb^1(\R)$ be such that $f(y) \stackrel{\infty}{\sim} e^{-\psi_1(y)}$ and $g(y) \stackrel{\infty}{\sim} e^{-\psi_2(y)}$, for some $\psi_1, \psi_2$ with $\lim_{y \to \infty} \psi_1'(y) = \lim_{y\to \infty} \psi_2'(y) = \infty$. Then $(f*g)(y) \stackrel{\infty}{\sim} (e^{-\psi_1}*e^{-\psi_2})(y)$.
\item \label{item-3:lem:tail-asymptotics} Let $f$ and $g$ be probability densities with Gaussian tails, that is $f(y) \stackrel{\infty}{\sim} \eta_1(y)e^{-\psi_1(y)}$ and $g(y) \stackrel{\infty}{\sim} \eta_2(y)e^{-\psi_2(y)}$, and suppose that we have $\lim_{y \to \infty} \psi_1'(y) = \lim_{y\to \infty} \psi_2'(y) = \infty$. Then $f*g$ has a Gaussian tail, i.e.~$(f*g)(y) \stackrel{\infty}{\sim} \eta_0(y)e^{-\psi_0(y)}$ for some $\psi_0 \in \ap$ and some $\eta_0$ flat with respect to $\psi_0$. Specifically, writing $y(u) = q_1 + q_2 = \psi_1'^{-1}(u) + \psi_2'^{-1}(u)$, we have
\begin{align*}
\psi_0(y) &= \psi_1(q_1) + \psi_2(q_2) \\
\eta_0(y) &= \frac{\sqrt{2\pi} s_{\psi_1}(q_1)\eta_1(q_1) s_{\psi_2}(q_2)\eta_2(q_2)}{\sqrt{s_{\psi_1}^2(q_1) + s_{\psi_2}^2(q_2)}}.
\end{align*}
In particular, for $d \geq 1$, the $d$-fold convolution of $f$ with itself $f^{*d}$ satisfies
\begin{equation*}
f^{*d}(y) \stackrel{\infty}{\sim} \frac{1}{\sqrt{d}} \left(\frac{2\pi}{\psi_1''\left(\frac{y}{d}\right)}\right)^{\frac{d-1}{2}} f\left(\frac{y}{d}\right)^d.
\end{equation*}
\end{enumerate}
\end{lemma}

Before giving the proof, we note that \Cref{item-2:lem:tail-asymptotics} of \Cref{lem:tail-asymptotics} gives conditions under which the asymptotics of the convolution of integrable functions can be identified from the asymptotics of the functions themselves. On the other hand, \Cref{item-3:lem:tail-asymptotics} states that Gaussian tails are closed under additive convolution and allows one to identify the asymptotic explicitly, this latter feature being particularly useful. The statement of \Cref{lem:tail-asymptotics}\ref{item-3:lem:tail-asymptotics} is the content of \cite[Theorem 1.1 and (1.11)]{balkema:1993}, and our aim, in incorporating it as an item of a lemma, is merely to improve the clarity and presentation of the proof of \Cref{thm:asymptotics-infinity}.

\begin{proof}
The first claim is straightforward. The proof of \Cref{item-2:lem:tail-asymptotics} is in the spirit of the proof of \cite[Proposition 2.2]{balkema:1993}. Since $f$ and $g$ are asymptotic to positive functions it follows that they are themselves eventually positive. This, and the other properties of $\psi_1$ and $\psi_2$, allows us to choose $a > 0$ large enough such that: (1) both $\psi_1$ and $\psi_2$ are well-defined on $(a,\infty)$, (2) $\psi'_1, \psi_2' > 0$ on $(a,\infty)$, (3) $\int_{-\infty}^a |g(y)|dy \neq 0$ and $\int_{-\infty}^a |f(y)|dy \neq 0$, and (4) $c_g = \int_{a+1}^{a+2} g(x) dx > 0$ and $c_f = \int_{a+1}^{a+2} f(x)dx > 0$. For $x > 2a$,
\begin{equation*}
(f*g)(x) = \int_a^{x-a} f(x-y)g(y)dy + \int_{-\infty}^a f(x-y)g(y)dy + \int_{-\infty}^a f(y)g(x-y)dy,
\end{equation*}
so by symmetry it suffices to show that $\int_{-\infty}^a e^{-\psi_1(x-y)}g(y)dy$ is of order $\littleo\left(\int_a^{x-a} f(x-y)g(y)dy\right)$ at infinity. Since $\psi_1' > 0$ on $(a,\infty)$
\begin{equation*}
\left|\int_{-\infty}^a e^{-\psi_1(x-y)}g(y)dy \right| \leq C e^{-\psi_1(x-a)}
\end{equation*}
with $C = \int_{-\infty}^a |g(y)|dy \neq 0$ a constant. By the mean value theorem,
\begin{align*}
\int_{a+1}^{a+2} e^{-\psi_1(x-y)} g(y)dy &\geq c_g e^{-\psi_1(x-a-1)} = c_g e^{-\psi_1(x-a)} e^{\psi_1'(x-z)} \\
&\geq \frac{c_g}{C} e^{\psi_1'(x-z)} \left|\int_{-\infty}^a e^{-\psi_1(x-y)}g(y)dy\right|,
\end{align*}
with $|z| \leq a+1$, and letting $x \to \infty$ finishes the proof of the second claim. Finally, \Cref{item-3:lem:tail-asymptotics} is the content of \cite[Theorem 1.1 and (1.11)]{balkema:1993}.
\end{proof}

\subsubsection{Proof of \Cref{thm:asymptotics-infinity}\ref{item-1:thm:asymptotics-infinity}} \label{subsec:asymptotics-infinity-proofs-1}

For convenience we write $\alpha$ in place of $\mathfrak{t}$ and thus our assumption is that $\phi$ is a Bernstein function such that $\phi(\infty) = \infty$ and $\phi^\alpha \in \B_\J$, for all $\alpha \in (0,1)$. We write $(\nu_t)_{t \geq 0}$ for the Berg-Urbanik semigroup associated to $\phi$ and, for any $\alpha \in (0,1)$, let $(\bar{\nu}_t)_{t > 0}$ denote the Berg-Urbanik semigroup associated to $\phi^\alpha$. Then, for $n \geq 0$ and any $\alpha \in (0,1)$, we have by the moment determinacy of any Berg-Urbanik semigroup up to time 2 that
\begin{equation*}
\M_{\bar{\nu}_1}(n) = \prod_{k=1}^n \phi^\alpha(k) = \left(\prod_{k=1}^n \phi(k)\right)^{\alpha } = \M_{\nu_{\alpha}}(n),
\end{equation*}
and applying \cite[Theorem 2.2]{berg:2007} then gives that $(\bar{\nu}_t)_{t \geq 0} = (\nu_{\alpha t})_{t \geq 0}$. Since $\phi^\alpha \in \B_\J$, for any $\alpha \in (0,1)$, and plainly $\phi(\infty) = \infty$ implies $\phi^\alpha(\infty) = \infty$, we conclude that $\mathtt{N}_{\phi^\alpha} = \infty$. Invoking \Cref{thm:smoothness-t}\ref{item-2:thm:smoothness-t} then yields, for any $t > 0$ and $\alpha \in (0,1)$, $\bar{\nu}_t \in \mathtt{C}_0^\infty(\R_+)$, from which we deduce that $\nu_t \in \mathtt{C}_0^\infty(\R_+)$, where $\nu_t(dx) = \nu_t(x)dx$, $x,t > 0$. Since $\phi^\alpha \in \B_\J$ with $\phi^\alpha(\infty) = \infty$ we may apply \cite[Theorem 5.5]{patie:2015} to obtain, for any $n \geq 0$, the asymptotic relation
\begin{equation*}
\bar{\nu}_1^{(n)}(x) = \nu_\alpha^{(n)}(x) \stackrel{\infty}{\sim} (-1)^n\frac{C_{\phi,\alpha}}{\sqrt{2\pi}} x^{-n} \varphi_\alpha^n(x) \sqrt{\varphi_\alpha'(x)}e^{-\int_{\sf{k}^\alpha}^x \frac{\varphi_\alpha(y)}{y} dy}
\end{equation*}
where $C_{\phi,\alpha} > 0$ is a constant depending only on $\phi$ and $\alpha$, $\varphi_\alpha:[\sf{k}^\alpha,\infty) \to [0,\infty)$ is the continuous inverse of the function $u \mapsto \phi^\alpha(u)$ and $\sf{k} = \phi(0)$. The constant $C_{\phi,\alpha}$ may be identified as $C_\phi^\alpha$, where $C_\phi > 0$ is a constant depending only on $\phi$, cf.~\cite[Theorem 5.1(2)]{patie:2015}, and plainly $\varphi_\alpha(u) = \varphi(u^{\frac{1}{\alpha}})$, where $\varphi:[\sf{k},\infty)\to [0,\infty)$ is the continuous inverse of $\phi$. Thus, by some routine calculations, we conclude that
\begin{equation}
\label{eq:asymptotic-nu-alpha}
\nu_\alpha^{(n)}(x) \stackrel{\infty}{\sim} (-1)^n \frac{C_\phi^\alpha}{\sqrt{2\pi \alpha}} x^{-n-\frac{1}{2}} \varphi^n(x^\frac{1}{\alpha}) \sqrt{x^{\frac{1}{\alpha}}\varphi'(x^\frac{1}{\alpha})} e^{-\alpha\int_\sf{k}^{x^\frac{1}{\alpha}} \frac{\varphi(r)}{r}dr}.
\end{equation}
Since $\alpha \in (0,1)$ is arbitrary this proves the claimed asymptotic for any $n \geq 0$ and $t \in (0,1)$.

We proceed by showing that for $n = 0$, i.e.~for the density $\nu_t(x)$ itself, the claimed asymptotic holds for all $t > 0$, and then extend this to the case when $n \geq 1$. To this end we define, for $y \in \R$ and $t > 0$, $f_t(y) = e^y\nu_t(e^y)$ and set $f_0 = \delta_0$. Then by \Cref{lem:tail-asymptotics}\ref{item-1:lem:tail-asymptotics} $(f_t)_{t \geq 0}$ is an additive convolution semigroup of probability densities, and from \eqref{eq:asymptotic-nu-alpha} together with some simple algebra we get, for $\alpha \in (0,1)$,
\begin{equation}
\label{eq:asymptotic-f-alpha}
f_\alpha(y) \stackrel{\infty}{\sim} \frac{C_\phi^\alpha}{\sqrt{2\pi \alpha}} e^{\frac{y}{2}} \sqrt{e^{\frac{y}{\alpha}}\varphi'(e^{\frac{y}{\alpha}})} e^{-\alpha\int_\sf{k}^{e^\frac{y}{\alpha}} \frac{\varphi(r)}{r}dr}.
\end{equation}
Let us write
\begin{equation*}
\bar{\psi}(y) = \int_\sf{k}^{e^y} \frac{\varphi(r)}{r}dr = \Leg_G(y) + \int_\sf{k}^{\phi(1)} \frac{\varphi(r)}{r} dr
\end{equation*}
where $\Leg_G$ is the Legendre transform of the function $G$ is defined in \eqref{eq:def-G}. From \Cref{lem:self-neglecting} we get that $\bar{\psi} \in \ap$, and writing $\psi$ for the function
\begin{equation} \label{eq:def-psi}
\psi(y) = \alpha\int_\sf{k}^{e^\frac{y}{\alpha}} \frac{\varphi(r)}{r}dr = \alpha \bar{\psi}\left(\frac{y}{\alpha}\right),
\end{equation}
we get from \Cref{lem:flat-big-O}\ref{item-5:lem:flat-big-O} that $\psi \in \ap$. A straightforward calculation gives that its scale function $s_\psi$ takes the form
\begin{equation*}
s_{\psi}(y) = \sqrt{\frac{\alpha}{e^\frac{y}{\alpha} \varphi'(e^\frac{y}{\alpha})}},
\end{equation*}
so combining \Cref{item-1:lem:flat-big-O,item-2:lem:flat-big-O} of \Cref{lem:flat-big-O} we get that $\sqrt{e^\frac{y}{\alpha}\varphi'(e^\frac{y}{\alpha})}$ is flat with respect to $\psi$. Furthermore, as $\phi'$ is non-increasing positive,  $\lim_{u \to \infty} \phi'(u) < \infty$ and thus we have
\begin{equation*}
\lim_{y \to \infty} s_{\psi}(y) = \sqrt{\alpha} \lim_{y\to \infty} \frac{1}{\sqrt{e^\frac{y}{\alpha} \varphi'(e^\frac{y}{\alpha})}} = \sqrt{\alpha} \lim_{y \to \infty} e^{-\frac{y}{2\alpha}} \sqrt{\phi'(\varphi(e^\frac{y}{\alpha}))} = 0.
\end{equation*}
Hence
\begin{equation*}
\lim_{y \to \infty} \exp\left(\frac{ws_{\psi}(y)}{2}\right) = 1, \quad \text{locally uniformly in } w \in \R,
\end{equation*}
which shows that $e^{\frac{y}{2}}$ is flat with respect to $\psi$. Constants are trivially flat with respect to $\psi$, so that putting all of these observations together we get that all the terms in front of the exponential in \eqref{eq:asymptotic-f-alpha} are flat with respect to $\psi$. Hence, for each $\alpha \in (0,1)$, $f_\alpha$ has a Gaussian tail.

Now we may invoke the second part of \Cref{lem:tail-asymptotics}\ref{item-3:lem:tail-asymptotics}, which states that the property of having a Gaussian tail is stable under additive convolution, to obtain for any $d \in \N$
\begin{equation*}
f_{d\alpha}(y) \stackrel{\infty}{\sim} \frac{1}{\sqrt{d}} \left(\frac{2\pi}{\psi''\left(\frac{y}{d}\right)}\right)^{\frac{d-1}{2}} f_{\alpha}\left(\frac{y}{d}\right)^d = \frac{1}{\sqrt{d}} \left(\frac{2\pi \alpha}{e^\frac{y}{\alpha d}\varphi'(e^\frac{y}{\alpha d})}\right)^{\frac{d-1}{2}} f_{\alpha}\left(\frac{y}{d}\right)^d.
\end{equation*}
Since for any $t > 0$ we can find $\alpha \in (0,1)$ and $d \in \N$ such that $t = \alpha d$ we get from the above relation the asymptotic of $f_t$ for all $t > 0$. Hence, after performing some straightforward computations and changing variables again, we get that for any $t > 0$,
\begin{equation}
\label{eq:asymptotic-nu-t-proof}
\nu_t(x) \stackrel{\infty}{\sim} \frac{C_\phi^t}{\sqrt{2\pi t}} \sqrt{x^{\frac{1-t}{t}} \varphi'(x^\frac{1}{t})} e^{-t\int_\sf{k}^{x^\frac{1}{t}} \frac{\varphi(r)}{r}dr},
\end{equation}
which proves the claim for $n =0$.

Next, suppose that $n \geq 1$. A straightforward application of the chain rule gives that $f_\alpha^{(n)}(y) = (e^y\nu_\alpha(e^y))^{(n)}$ is a linear combination of terms of the form $e^{(k+1)y}\nu_\alpha^{(k)}(e^y)$, for $0 \leq k \leq n$. However, from \eqref{eq:asymptotic-nu-alpha} we deduce that, for large $y$, the term $e^{(n+1)y}\nu_\alpha^{(n)}(e^y)$ grows faster than all terms of lower order. Therefore,
\begin{equation}
\label{eq:asymptotic-f-n-alpha}
f_\alpha^{(n)}(y) \stackrel{\infty}{\sim} e^{(n+1)y} \nu_\alpha^{(n)}(e^y) \stackrel{\infty}{\sim}  (-1)^n \frac{C_\phi^\alpha}{\sqrt{2\pi \alpha}} e^{\frac{y}{2}} \varphi^n(e^{\frac{y}{\alpha}}) \sqrt{e^\frac{y}{\alpha} \varphi'(e^\frac{y}{\alpha}) } e^{-\alpha\int_\sf{k}^{e^\frac{y}{\alpha}} \frac{\varphi(r)}{r} dr}
\end{equation}
and the asymptotic on the right-hand side is obtained from the one in \eqref{eq:asymptotic-nu-alpha} after changing variables. From the right-hand side of \eqref{eq:asymptotic-f-n-alpha} it is apparent that the mapping  $y\mapsto (-1)^nf_\alpha^{(n)}(y)$ is eventually positive, so that there exists $a_n \in \R$ (depending on $n$) such that $\sf{f}_{\alpha,n}(y) = (-1)^nf_\alpha^{(n)}(y)\mathbb{I}_{\{y > a_n\}}$ is a positive function. Since $y \mapsto \varphi(e^\frac{y}{\alpha})$ is the derivative of $\psi$, which we recall from earlier denotes the function appearing within the exponential in \eqref{eq:asymptotic-f-n-alpha}, we have from \cite[Proposition 5.8]{balkema:1993} that $y \mapsto \varphi(e^\frac{y}{\alpha})$ is flat with respect to $\psi$, and combined with \Cref{lem:flat-big-O}\ref{item-3:lem:flat-big-O} this gives that $y \mapsto \varphi^n(e^\frac{y}{\alpha})$ is flat with respect to $\psi$. Thus, once again all terms in front of the exponential in \eqref{eq:asymptotic-f-n-alpha} are flat with respect to $\psi$. Let $\epsilon \in (0,\alpha)$ so that, from \Cref{lem:flat-big-O}\ref{item-4:lem:flat-big-O} applied to \eqref{eq:asymptotic-f-alpha}, we deduce the estimate
\begin{equation}
\label{eq:f_c-big-O}
\sf{f}_{\alpha,n}(y) \stackrel{\infty}{=} \bigO\left(e^{-(\alpha - \epsilon)\int_\sf{k}^{e^\frac{y}{\alpha}} \frac{\varphi(r)}{r} dr}\right).
\end{equation}
Then \eqref{eq:asymptotic-f-alpha} allows us to identify the right-hand side of \eqref{eq:f_c-big-O} as the dominant term in the asymptotic for the probability density $f_{\alpha - \epsilon}(y) = e^y\nu_{\alpha - \epsilon}(e^y)$, see \eqref{eq:asymptotic-f-alpha}. Indeed, the fact the function inside the big-$\bigO$ estimate of \eqref{eq:f_c-big-O} term dominates all others in \eqref{eq:asymptotic-f-alpha} is immediate, as the term in front of the exponential is increasing at infinity. Noting that dilating a function does not affect its integrability, we conclude that, for any $\alpha \in (0,1)$ and $n \geq 1$, the function $\sf{f}_{\alpha,n}$ is integrable. In particular, for each $\alpha \in (0,1)$ and $n \geq 1$ there exists a constant $c_{\alpha,n} > 0$ such that $c_{\alpha,n} \sf{f}_{\alpha,n}$ is a probability density.

Now, let us write $t = \alpha + \tau$, where $\alpha \in (0,1)$ and $\tau > 0$. If, for any $n \geq 0$, $f_\alpha^{(n)} \in \Leb^2(\R)$, and $f_\tau \in \Leb^2(\R)$, then a standard result (see~\cite[Chapter 8, Ex.~8 \& 9]{folland:1999}) allows us to interchange differentiation and convolution to write that
\begin{equation}
\label{eq:convolution-representation}
f_t^{(n)}(y) = (f_\alpha^{(n)} * f_\tau)(y), \: y \in \R.
\end{equation}
To this end, let $\bm{t} > 0$ and observe that
\begin{equation}
\label{eq:L-2-finiteness}
\int_{-\infty}^\infty \left(e^{(n+1)y}\nu_{\bm{t}}^{(n)}(e^y)\right)^2 dy = \int_0^\infty \left(x^{n+\frac{1}{2}} \nu_{\bm{t}}^{(n)}(x)\right)^2 dx = \frac{1}{2\pi} \int_{-\infty}^{\infty} \frac{|\Gamma(1+n+ib)|^2}{|\Gamma(1+ib)|^2} \left|W_\phi^{\bm{t}}\left(1+ib \right)\right|^2db
\end{equation}
where the first equality follows from a change of variables, and the second is a combination of the Parseval formula for the Mellin transform applied to the function $x\mapsto x^{n+\frac{1}{2}} \nu_{\bm{t}}^{(n)}(x)$ combined with \Cref{thm:Bernstein-Gamma-Mellin-transform}. By \cite[Theorem 4.2(3)(c)]{patie:2016}, the fact that $\phi(\infty) = \infty$ with $\phi^\alpha \in \B_\J$ implies that $b \mapsto |W_\phi^{\bm{t}}(1+ib)|$ decays faster than any polynomial along the real line. Next, we recall Stirling's formula for the gamma function, for any $a + ib$ with $a > 0$ fixed
\begin{equation}\label{eq:stirling}
|\Gamma(a+ib)| \stackrel{\infty}{\sim} C_a |b|^{a-\frac{1}{2}} e^{-\frac{\pi}{2}|b|}
\end{equation}
for some constant $C_a > 0$. Hence, the term in \eqref{eq:L-2-finiteness} involving the ratio of  gamma functions grows like $|b|^{2n+2}$, which by the aforementioned decay properties of $W_\phi^{\bm{t}}$ gives that the integral in \eqref{eq:L-2-finiteness} is finite. Since $f_\tau^{(n)} (y)= (e^y\nu_\alpha(e^y))^{(n)}$ is a linear combination of functions of the form $e^{(k+1)y}\nu_\alpha^{(k)}(e^y)$, for $k \leq n$, we get that $f_\alpha^{(n)} \in \Leb^2(\R)$ for any $n \geq 0$, and that $f_\tau \in \Leb^2(\R)$. Hence the equality in \eqref{eq:convolution-representation} is justified.

Next we aim to use a combination of \Cref{lem:tail-asymptotics}\ref{item-2:lem:tail-asymptotics} together with \eqref{eq:convolution-representation} in order to show that $f_t^{(n)}$ has a Gaussian tail. From \eqref{eq:asymptotic-f-n-alpha} we have
\begin{equation*}
(-1)^n f_\alpha^{(n)}(y) \stackrel{\infty}{\sim} h(y)e^{-\psi(y)}
\end{equation*}
where the function $\psi$ is defined in \eqref{eq:def-psi}, and $h$  denotes the function consisting of all terms in front of the exponential of \eqref{eq:asymptotic-f-n-alpha}. Since $h$ is flat with respect to $\psi$ we know, by \cite[Proposition 3.2]{balkema:1993}, that there exists $\chi \in \mathtt{C}^\infty(\R)$ such that $\chi (y) \stackrel{\infty}{\sim} h(y)$ and $s_{\psi}(y)\chi'(y) \stackrel{\infty}{=} \littleo(\chi(y))$. Further, from Proposition 5.8 in the aforementioned paper $\lim_{y \to \infty} s_{\psi}(y)\psi'(y) = \infty$. Using these facts we get
\begin{equation*}
\lim_{y \to \infty} \frac{(\log \chi(y))'}{\psi'(y)} = \lim_{y \to \infty} \frac{\chi'(y)}{\chi(y)\psi'(y)} = \lim_{y \to \infty} \frac{s_{\psi}(y)\chi'(y)}{\chi(y)} \frac{1}{s_{\psi}(y)\psi'(y)} = 0,
\end{equation*}
which is enough to show that $f_\alpha^{(n)}$ satisfies the assumptions of \Cref{lem:tail-asymptotics}\ref{item-2:lem:tail-asymptotics}. Since the arguments for $f_\tau$ are similar we have, invoking \Cref{lem:tail-asymptotics}\ref{item-2:lem:tail-asymptotics}, that
\begin{equation*}
(-1)^n c_{\alpha,n} f_t^{(n)}(y) \stackrel{\infty}{\sim} (c_{\alpha,n} \sf{f}_{\alpha,n} * f_\tau)(y),
\end{equation*}
with both $c_{\alpha,n}\sf{f}_{\alpha,n}$ and $f_\tau$ having Gaussian tails. Applying \Cref{lem:tail-asymptotics}\ref{item-3:lem:tail-asymptotics} again we conclude that  $c_{\alpha,n}\sf{f}_{\alpha,n} * f_\tau$ has a Gaussian tail, and hence $f_t^{(n)}(y) \stackrel{\infty}{\sim} (-1)^n \eta_0(y) e^{-\psi_0(y)}$, where $\psi_0 \in \ap$ and $\eta_0$ is flat with respect to $\psi_0$.

To conclude the proof it remains to identify $\eta_0$ and $\psi_0$, which may be computed as described in \Cref{lem:tail-asymptotics}\ref{item-3:lem:tail-asymptotics}, using a combination of \eqref{eq:asymptotic-f-n-alpha} and, after changing variables, \eqref{eq:asymptotic-nu-t-proof}. As in the lemma, we write $y(u) = q_1(u) + q_2(u) = \alpha \log \phi(u) + \tau\log \phi(u) = t\log \phi(u)$, where the second equality serves as definition of $q_1$ and $q_2$, and the last equality defines the conjugate variables $y$ and $u$. Using this notation it is straightforward to conclude that
\begin{equation*}
\psi_0(y) = \alpha \int_{\sf{k}}^{e^\frac{q_1}{\alpha}} \frac{\varphi(r)}{r}dr + \tau\int_{\sf{k}}^{e^{\frac{q_2}{\tau}}} \frac{\varphi(r)}{r}dr = (\alpha + \tau) \int_{\sf{k}}^{\phi(u)} \frac{\varphi(r)}{r}dr  = t\int_{\sf{k}}^{e^{\frac{y}{t}}} \frac{\varphi(r)}{r}dr.
\end{equation*}
The associated scale function $s_{\psi_0}$ is then
\begin{equation*}
s_{\psi_0}(y) = \sqrt{\frac{t}{e^\frac{y}{t}\varphi'(e^\frac{y}{t})}}.
\end{equation*}
Let $\eta_1$ and $\eta_2$ denote the flat terms, while $\psi_1$ and $\psi_2$ denote the asymptotically parabolic terms, in the Gaussian tails of $c_{\alpha,n}\sf{f}_{\alpha,n}$ and $f_\tau$ respectively. Then,
\begin{equation*}
\eta_1(q_1(u)) = \frac{C_\phi^\alpha}{\sqrt{2\pi \alpha}} (\phi(u))^\frac{\alpha}{2} u^n \sqrt{\phi(u)\varphi'(\phi(u))},
\end{equation*}
and
\begin{equation*}
\eta_2(q_2(u)) = \frac{C_\phi^\tau}{\sqrt{2\pi \tau}} (\phi(u))^\frac{\tau}{2} \sqrt{\phi(u)\varphi'(\phi(u))}.
\end{equation*}
Furthermore,
\begin{equation*}
s_{\psi_1}(q_1(u)) = \sqrt{\frac{\alpha}{\phi(u)\varphi'(\phi(u))}} \quad \text{and} \quad s_{\psi_2}(q_2(u)) = \sqrt{\frac{\tau}{\phi(u)\varphi'(\phi(u))}}
\end{equation*}
where $s_{\psi_1}$ and $s_{\psi_2}$ are the scale functions of $\psi_1$ and $\psi_2$, respectively. Putting all of these observations together we get that $\eta_0$ can be written, after canceling like terms, as
\begin{equation*}
\eta_0(y) = \frac{C_\phi^{(\alpha+\tau)}\sqrt{2\pi}}{\sqrt{2\pi\alpha}\sqrt{2\pi\tau}} e^\frac{y}{2} \sqrt{\alpha}\sqrt{\tau} \frac{\sqrt{e^\frac{y}{t}\varphi'(e^\frac{y}{t})}}{\sqrt{t}} \varphi^n(e^\frac{y}{t}) = \frac{C_\phi^t}{\sqrt{2\pi t}} e^\frac{y}{2}\varphi^n(e^\frac{y}{t})\sqrt{e^\frac{y}{t}\varphi'(e^\frac{y}{t})}.
\end{equation*}
This gives us $f_t^{(n)}(y) \stackrel{\infty}{\sim} (-1)^n\eta_0(y)e^{-\psi_0(y)} \stackrel{\infty}{\sim} e^{(n+1)y}\nu_t^{(n)}(e^y)$, and changing variables again, we finally obtain the claimed asymptotic
\begin{equation*}
\nu_t^{(n)}(x) \stackrel{\infty}{\sim} (-1)^n \frac{C_\phi^t}{\sqrt{2\pi t}} x^{-n} \varphi^n(x^\frac{1}{t}) \sqrt{x^{\frac{1-t}{t}}\varphi'(x^\frac{1}{t})} e^{-t\int_\sf{k}^{x^\frac{1}{t}} \frac{\varphi(r)}{r}dr},
\end{equation*}
for any $n \geq 0$ and $t > 0$, which completes the proof.

\subsubsection{Proof of \Cref{thm:asymptotics-infinity}\ref{item-1:thm:asymptotics-infinity} and \Cref{thm:asymptotics-infinity}\ref{item-2:thm:asymptotics-infinity}} \label{subsec:asymptotics-infinity-proofs-2}
 The proof is the same for \cite[Theorem 5.5(1)]{patie:2015} and \cite[Theorem 5.5(2)]{patie:2015}, but we give the arguments for sake of completeness.  Suppose that $\sf{d} > 0$, so that
\begin{equation*}
\phi(u) = \sf{k} + \sf{d}u + u\int_0^\infty e^{-uy} \bar{\mu}(y)dy.
\end{equation*}
Then, invoking \cite[Proposition 4.1(3)]{patie:2015} we have $\phi(u) \stackrel{\infty}{\sim} \sf{d}u$ and hence $\varphi(u) \stackrel{\infty}{\sim} \sf{d}^{-1} u$. Furthermore, differentiating the identity $u = \phi(\varphi(u))$ gives $\varphi'(u) = \frac{1}{\phi'(\varphi(u))}$ and since, by the monotone density theorem, see~\cite[Theorem 1.7.2]{bingham:1989}, $\phi'(u) \stackrel{\infty}{\sim} \sf{d}$, we get that $\varphi'(u) \stackrel{\infty}{\sim} \sf{d}^{-1}$. Next, as $u = \phi(\varphi(u))$ we have, on $[\sf{k},\infty)$,
\begin{equation*}
u = \sf{k} + \sf{d}\varphi(u) + \varphi(u)\int_0^\infty e^{-\varphi(u)y} \bar{\mu}(y)dy = \sf{k} + \sf{d}\varphi(u) + E(u)
\end{equation*}
where the last equality serves to define the function $E$. By dominated convergence we have that \\ $\lim_{u \to \infty} \int_0^\infty e^{-\varphi(u)y} \bar{\mu}(y)dy = 0$ which, together with $\varphi(u) \stackrel{\infty}{\sim} \sf{d}^{-1} u$, shows that $E(u) = \littleo(u)$. Re-arranging, we obtain $\varphi(y) =\sf{d}^{-1}(u - \sf{k} - E(u))$, so that substituting all of these quantities into the identities \eqref{eq:asymptotic-nu-t} and \eqref{eq:asymptotic-nu-t-n} proves \Cref{item-1:thm:asymptotics-infinity}.

Next, assume that $\phi(u) \stackrel{\infty}{\sim} C_\alpha u^\alpha$, with $C_\alpha > 0$ a constant and $\alpha \in (0,1)$. A standard result from regular variation theory gives that $\varphi(u)  \stackrel{\infty}{\sim} C_\alpha^{-\frac{1}{\alpha}} u^\frac{1}{\alpha}$, see e.g.~\cite[Theorem 1.5.12]{bingham:1989}. This allows us to define $H(u) = C_\alpha^{-\frac{1}{\alpha}} u^\frac{1}{\alpha} - \varphi(u)$, so that $H(u) = \littleo(u^\frac{1}{\alpha})$. Next, the monotonicity of $\phi'$ allows us to again invoke the monotone density theorem to conclude that $\phi'(u) \stackrel{\infty}{\sim} C_\alpha\alpha u^{\alpha-1}$, see again~\cite[Theorem 1.7.2]{bingham:1989}. Combining these two statements with the identity
$\varphi'(u) = \frac{1}{\phi'(\varphi(u))}$ yields the asymptotic $\varphi'(u)  \stackrel{\infty}{\sim} \alpha^{-1}C_\alpha^{-\frac{1}{\alpha}} u^{\frac{1}{\alpha}-1}$. Finally, substituting these asymptotics proves the claim.
\subsection{Proofs for \Cref{subsec:threshold}} \label{subsec:threshold-result-proofs}

Before beginning with the proofs we state some preliminary results that will be used in the proof of \Cref{thm:threshold-result}\ref{item-2:thm:threshold-result} and \Cref{thm:threshold-result}\ref{item-3:thm:threshold-result}.

\begin{proposition}
\label{prop:reference-function}
For $\alpha \in (0,1)$ and $\mathfrak{m} \geq 0$, let $\phi_{\alpha,\mathfrak{m}}:[0,\infty) \to [0,\infty)$ be defined by $\phi_{\alpha,\mathfrak{m}}(u) = (u+\mathfrak{m})^\alpha$.
\begin{enumerate}
\item \label{item-1:prop:reference-function} For any $\alpha \in (0,1)$ and $\mathfrak{m} \geq 0$, $\phi_{\alpha,\mathfrak{m}}$ is a complete Bernstein function.
\item The potential measure of $\phi_{\alpha,\mathfrak{m}}$ admits a density, denoted by $U_{\alpha,\mathfrak{m}}$, given by
\begin{equation*}
U_{\alpha,\mathfrak{m}}(y) = \frac{1}{\Gamma(\alpha)} e^{-\mathfrak{m} y} y^{\alpha - 1}.
\end{equation*}
Furthermore, $U_{\alpha,\mathfrak{m}}$ is non-increasing, convex and solves, on $\R_+$, the differential equation
\begin{equation*}
U_{\alpha,\mathfrak{m}}' = -U_{\alpha,\mathfrak{m}}(y)\left(\mathfrak{m} + \frac{1-\alpha}{y}\right).
\end{equation*}
\item \label{item-3:prop:reference-function} Let $\phi \in \B_{d}$, i.e.~$\sf{d} > 0$. Then, for any $\alpha \in (0,1)$,
\begin{equation*}
y_\alpha = \inf\{y \geq 0; y\bar{\mu}(y) > \sf{d}(1-\alpha) \} \in (0,\infty],
\end{equation*}
and, for any $\mathfrak{m}$ such that $\sf{d}\mathfrak{m} \geq \bar{\mu}(\frac{y_\alpha}{2})+\sf{k}$, we have that $\frac{\phi}{\phi_{\alpha,\mathfrak{m}}} \in \B$.
\end{enumerate}
\end{proposition}

\begin{proof}
Let $\alpha \in (0,1)$ and $\mathfrak{m} \geq 0$. The fact that $\phi_{\alpha,\mathfrak{m}}$ is a complete Bernstein function is straightforward and was also mentioned in \Cref{rem:cb}. To show that $U_{\alpha,\mathfrak{m}}$ defined as above is the density of the potential measure of $\phi_{\alpha,\mathfrak{m}}$ we observe that
\begin{equation*}
\frac{1}{u^\alpha} = \frac{1}{\Gamma(\alpha)} \int_0^\infty e^{-uy} y^{\alpha-1}dy,
\end{equation*}
and then substitute $u+\mathfrak{m}$ for $u$. The claimed properties of $U_{\alpha,\mathfrak{m}}$ can then be verified by straightforward calculations. The proof of the last claim is, mutatis mutandis, the same as the one given for \cite[Proposition 4.4(2)]{patie:2015}, so we omit it here. The interested reader will find the details outlined in the PhD thesis of the second author, which is forthcoming.  Note that the proof of \cite[Proposition 4.4]{patie:2015} does not explicitly use the fact that the L\'evy measure of $\phi$ has a non-increasing density, and hence this restriction can be removed. Furthermore, we have modified $y_\alpha$ and the condition on $\mathfrak{m}$ to suit our potential measure $U_{\alpha,\mathfrak{m}}$.
\end{proof}

We write, for two functions $f$ and $g$, $f(x) \stackrel{\infty}{\asymp} g(x)$ if $f(x) \stackrel{\infty}{=} \bigO(g(x))$ and $g(x) \stackrel{\infty}{=} \bigO(f(x))$. In the following theorem we rephrase, in the context of Berg-Urbanik semigroups, an Abelian type criterion for moment indeterminacy that was given in \cite{patie:2018a}, which we use in the proof of \Cref{thm:threshold-result}\ref{item-3:thm:threshold-result}.

\begin{theorem}[Theorem 1.2(2) in \cite{patie:2018a}]
\label{thm:Abelian-theorem}
Let $(\nu_t)_{t \geq 0}$ be a Berg-Urbanik semigroup and suppose that, for some $t > 0$, $\nu_t(dx) = \nu_t(x)dx$, $x > 0$, and
\begin{equation*}
\nu_t(x) \stackrel{\infty}{\asymp} e^{-G(\log x)},
\end{equation*}
with $G \in \ap$ satisfying $\lim_{y \to \infty} G'(y) e^{-\frac{y}{2}} < \infty$. Then, writing $\gamma$ for the inverse of the continuous, increasing function $G'$,
\begin{equation*}
\sum_{n=n_0}^\infty e^{-\frac{\gamma(n)}{2}} < \infty, \enskip \text{ for some } \enskip n_0 \geq 1 \quad \iff \quad \nu_t \text{ is moment indeterminate.}
\end{equation*}
\end{theorem}

\subsubsection{Proof of \Cref{thm:threshold-result}\ref{item-1:thm:threshold-result}} \label{subsubsec:item-1}

First, invoking \cite[Theorem 5.1(2)]{patie:2015} we get
\begin{equation*}
\M_\nu(n) = W_\phi(n+1) \stackrel{\infty}{\sim} C_\phi \sqrt{\phi(n)}e^{G(n)}
\end{equation*}
where $G(n) = \int_1^n \log \phi(r) dr$ and $C_\phi > 0$ is a constant depending only on $\phi$. Integrating $G$ by parts, for any $t > 0$ and $n \geq 1$, gives us
\begin{equation*}
\frac{t}{2n}G(n) = \frac{t}{2}\log \phi(n) - \frac{t}{2n}\left(\log \phi(1) + \int_1^n u\frac{\phi'(u)}{\phi(u)}du \right).
\end{equation*}
Consequently, for some $C_1 > 0$ a constant, we have
\begin{equation}
\label{eq:exponential-bound}
\sum_{n=1}^\infty W_\phi^{-\frac{t}{2n}}(n+1) \geq C_1 \sum_{n=1}^\infty \exp\left[-\frac{t}{2}\left(\log \phi(n)+ \frac{1}{2n}\log \phi(n)\right)\right] \exp\left[\frac{t}{2n}\left(\log \phi(1) + \int_1^n u\frac{\phi'(u)}{\phi(u)}du \right)\right].
\end{equation}
The estimate $\phi(n) \stackrel{\infty}{=} \bigO(n)$, see e.g.~\cite[Proposition 4.1(3)]{patie:2015}, gives $\log \phi(n) \stackrel{\infty}{=} \littleo(n)$, which together with the positivity of the terms within the second exponential in \eqref{eq:exponential-bound} allows us to obtain, for $C_2 > 0$ a constant, the bound
\begin{equation*}
\sum_{n=1}^\infty W_\phi^{-\frac{t}{2n}}(n+1) \geq C_1e^{-C_2t} \sum_{n=1}^\infty \phi^{-\frac{t}{2}}(n),
\end{equation*}
so to prove moment determinacy it suffices to show the divergence of this latter series. Let $\beta > \beta_\phi$. By definition of $\beta_\phi$, $\phi(u) \stackrel{\infty}{=} \bigO(u^\beta)$, so that for some constant $C_3 > 0$
\begin{equation*}
\sum_{n=1}^\infty \phi^{-\frac{t}{2}}(n) \geq C_3 \sum_{n=1}^\infty n^{-\frac{t\beta}{2}}.
\end{equation*}
The latter series diverges if and only if $t \beta \leq 2$, whence the moment determinacy of $\nu_t$ for any $t \leq \frac{2}{\beta} < \frac{2}{\beta_\phi}$. Since $\beta > \beta_\phi$ is arbitrary we conclude that $\scr{T}_\phi \geq \frac{2}{\beta_\phi}$ if $\beta_\phi > 0$ and $\scr{T}_\phi = \infty$ for $\beta_\phi = 0$. Finally, if $\limsup_{u \to \infty} u^{-\beta_\phi}\phi(u) < \infty$ then we may choose $\beta = \beta_\phi$ and apply the above argument to conclude that $\nu_{\scr{T}_\phi}$ is moment determinate.

\subsubsection{Proof of \Cref{thm:threshold-result}\ref{item-3:thm:threshold-result}}

It suffices to treat the case when $\delta_\phi \in (0,1]$, since otherwise the claimed right-hand inequality in \eqref{eq:inequality-T} is trivial. Therefore we assume also that $0 < \delta_\phi \leq \beta_\phi \leq 1$, and $\delta_\phi > 0$ is easily seen to imply that $\phi(\infty)=\infty$. Invoking \Cref{thm:asymptotics-infinity} we get that, for any $t > 0$,
\begin{equation*}
\nu_t(x) \stackrel{\infty}{\sim} \frac{C_\phi^t}{\sqrt{2\pi t}} x^{\frac{1-t}{2t}} \sqrt{\varphi'(x^\frac{1}{t})} e^{-t\int_\sf{k}^{x^\frac{1}{t}} \frac{\varphi(r)}{r}dr}.
\end{equation*}
Let $b(\log x)$ denote all the terms in front of the exponential and set $\bar{G}(\log x)$ for the function within the exponential on the right-hand of the above asymptotic relation. It was shown in the proof of \Cref{thm:asymptotics-infinity} that $b$ is flat with respect to $\bar{G}$, and thus, by \Cref{lem:flat-big-O}\ref{item-4:lem:flat-big-O} we have that $b(\log x) \stackrel{\infty}{=} \littleo(\bar{G}(\log x))$. Hence, for any $c \in (0,t)$ fixed we get that
\begin{equation*}
\nu_t(x) \stackrel{\infty}{\asymp} e^{-G(\log x)}
\end{equation*}
where
\begin{equation*}
G(\log x) = (t-c) \int_\sf{k}^{x^\frac{1}{t}} \frac{\varphi(r)}{r}dr.
\end{equation*}
From \Cref{lem:self-neglecting} it follows that $G \in \ap$ and a simple calculation, after substituting $y = \log x$, gives that
\begin{equation*}
G'(y) = \frac{(t-c)}{t} \varphi(e^\frac{y}{t}) = \mathfrak{t}\varphi(e^\frac{y}{t})
\end{equation*}
where we write $\bm{t} = \frac{(t-c)}{t} \in (0,1)$ for ease of notation. Observe that, for any $\delta \in (0,\delta_\phi)$, the property $\liminf_{u \to \infty} u^{-\delta}\phi(u) > 0$ is equivalent to $\limsup_{u \to \infty} u^{-\frac{1}{\delta}} \varphi(u) < \infty$. Hence, for any $\delta > \delta_\phi$ and $t \geq \frac{2}{\delta}$ we have
\begin{equation*}
\lim_{y \to \infty} G'(y) e^{-\frac{y}{2}} = \mathfrak{t} \lim_{y \to \infty} \varphi(e^\frac{y}{t}) e^{-\frac{y}{2}} = \bm{t} \lim_{y \to \infty} e^{-\frac{y}{\delta t}}\varphi(e^\frac{y}{t}) e^{\left(\frac{1}{\delta t}-\frac{1}{2}\right)y} < \infty,
\end{equation*}
and thus all the assumptions of \Cref{thm:Abelian-theorem} are fulfilled for any $t > \frac{2}{\delta_\phi}$. The inverse of $G'$ is easily identified as $\gamma(u) = t\log\phi(\bm{t}u)$ so that,
\begin{equation}
\label{eq:series}
\sum_{n=1}^\infty e^{-\frac{\gamma(n)}{2}} = \sum_{n=1}^\infty \phi^{-\frac{t}{2}}(\bm{t}n).
\end{equation}
Now, for any $\delta \in (0, \delta_\phi)$, there exists a constant $C > 0$ (depending only on $t$) such that, for $n$ large enough,
\begin{equation*}
\phi^{-\frac{t}{2}}(\bm{t}n) \leq C n^{-\frac{\delta t}{2}}.
\end{equation*}
Thus for any $t > \frac{2}{\delta}$ the series in \eqref{eq:series} converges, so that $\nu_t$ is indeterminate. Since $\delta$ can be taken arbitrarily close to $\delta_\phi$ this gives the indeterminacy of $\nu_t$ for any $t > \frac{2}{\delta_\phi}$.

\subsubsection{Proof of \Cref{thm:threshold-result}\ref{item-4:thm:threshold-result}}

Let $\frac{\phi}{\vartheta} \in \B$ and write $(\rho_t)_{t  > 0}$ for the Berg-Urbanik semigroup associated to $\vartheta$. Since $\frac{\phi}{\vartheta} \in \B$ we may invoke \cite[Theorem 4.7(3)]{patie:2016} to get that, for any $t > 0$ and $n \geq 0$,
\begin{equation*}
W_\phi^t(n+1) = W_{\frac{\phi}{\vartheta}}^t(n+1) W_\vartheta^t(n+1)
\end{equation*}
where each of the terms is a moment sequence. Applying \cite[Lemma 2.2 and Remark 2.3]{berg:2004} we conclude that whenever $\rho_t$ is indeterminate then $\nu_t$ is indeterminate, i.e.
\begin{equation*}
\{t>0;\rho_t \text{ is indeterminate}\} \subseteq \{t>0;\nu_t \text{ is indeterminate}\},
\end{equation*}
which implies that $\scr{T}_\phi \leq \scr{T}_\vartheta$. If $\vartheta^\mathfrak{t} \in \B_\J$ for all $\mathfrak{t} \in (0,1)$, then invoking \Cref{thm:threshold-result}\ref{item-3:thm:threshold-result} yields $\scr{T}_\phi \leq \frac{2}{\delta_\vartheta}$, which completes the proof.

\subsubsection{Proof of \Cref{thm:threshold-result}\ref{item-2:thm:threshold-result}}

First, by \Cref{prop:reference-function} and using the notation therein, we have for any $\alpha \in (0,1)$ and $\mathfrak{m} \geq \frac{\bar{\mu}(\frac{y_\alpha}{2})+\sf{k}}{\sf{d}}$ that $\frac{\phi}{\phi_{\alpha,\mathfrak{m}}} \in \B$. Hence, by \Cref{thm:threshold-result}\ref{item-4:thm:threshold-result} it follows that $\scr{T}_\phi \leq \scr{T}_{\phi_{\alpha,\mathfrak{m}}}$. \Cref{prop:reference-function}\ref{item-1:prop:reference-function} gives that $\phi_{\alpha,\mathfrak{m}}$ is a complete Bernstein function so that $\phi_{\alpha,\mathfrak{m}}^\mathfrak{t} \in \B_\J$ for all $\mathfrak{t} \in (0,1)$, see e.g.~\Cref{rem:cb}. Plainly $\phi_{\alpha,\mathfrak{m}} \stackrel{\infty}{\sim} u^\alpha$, which implies that $\delta_{\phi_{\alpha,\mathfrak{m}}} = \beta_{\phi_{\alpha,\mathfrak{m}}} = \alpha$. Invoking \Cref{thm:threshold-result}\ref{item-3:thm:threshold-result} we get that $\scr{T}_{\phi} \leq \frac{2}{\alpha}$. Since this inequality holds for any $\alpha \in (0,1)$ we get $\scr{T}_{\phi} \leq 2$, whence $\scr{T}_{\phi} = 2$. The claim that $\nu_2$ is moment determinate follows from \Cref{thm:threshold-result}\ref{item-1:thm:threshold-result}, since $\sf{d} > 0$ implies $\beta_\phi = 1$ and that $\limsup_{u \to \infty} u^{-1} \phi(u) = \lim_{u \to \infty} u^{-1}\phi(u) = \sf{d}$.

\subsection{Proofs for \Cref{subsec:lin-conjecture}} \label{subsec:lin-conjecture-proofs}

In the proofs below we write, for any $\phi \in \B$, $X(\phi) = X$ for the positive random variable whose law is $\nu_1^\phi$, and, for any $x, t > 0$, $\sigma_t(dx) = \mathbb{P}(X^t \in dx)$.

\subsubsection{Proof of \Cref{thm:Lin-conjecture}\ref{item-1:thm:Lin-conjecture}}

From \cite[Theorem 5.1(2)]{patie:2015} it follows that, for $t > 0$,
\begin{equation*}
\E\left[ \left(X^t \right)^n\right] = W_\phi(tn+1) \stackrel{\infty}{\sim} C_\phi \sqrt{\phi(tn)}e^{G(tn)}
\end{equation*}
where $G(tn) = \int_1^{tn} \log \phi(r) dr$ and $C_\phi > 0$ is a constant depending only on $\phi$. By following similar arguments than the ones developed for  the proof of \Cref{thm:threshold-result}\ref{item-1:thm:threshold-result} we obtain the estimate
\begin{equation*}
\sum_{n = 1}^\infty W_\phi^{-\frac{1}{2n}}(tn+1) \geq C \sum_{n=1}^\infty \phi^{-\frac{t}{2}}(tn),
\end{equation*}
for some constant $C > 0$. Now, for any $\beta > \beta_\phi$, $\phi(n) \stackrel{\infty}{=} \bigO(n^\beta)$ so that, for a constant $C_1 > 0$ depending on $t$ and $\beta$,
\begin{equation*}
\sum_{n=1}^\infty \phi^{-\frac{t}{2}}(tn) \geq C_1 \sum_{n=1}^\infty n^{-\frac{\beta t}{2}}.
\end{equation*}
This latter series diverges if and only if $t \leq \frac{2}{\beta} < \frac{2}{\beta_\phi}$, so that by Carleman's criterion $X^t$ is moment determinate whenever $t < \frac{2}{\beta_\phi}$. When $\limsup_{u \to \infty} u^{-\beta_\phi}\phi(u) < \infty$ we may take $\beta = \beta_\phi$ and apply the above argument, which finishes the proof.

\subsubsection{Proof of \Cref{thm:Lin-conjecture}\ref{item-3:thm:Lin-conjecture}}

Observe that, for $z \in 1 + i\R$ and $t > 0$, we have
\begin{equation*}
\M_{\sigma_t}(z-1) = \E[\left(X_1^t(\phi)\right)^{z-1}] = \E[X^{t(z-1)}] = W_\phi(tz-t+1).
\end{equation*}
Since $W_\phi \in \A_{(0,\infty)}$ it follows that $\M_{\sigma_t}(z-1)$ can be analytically extended to $\Re(z) > 1-\frac{1}{t}$, and we write $\M_{\sigma_t}$ for this analytical extension. Next, we may assume that $\delta_\phi > 0$, since the claim is trivial otherwise, from which it follows that $\phi(\infty)=\infty$. Combining this with the fact that $\phi \in \B_\mc{J}$ gives $\texttt{N}_\phi = \infty$, where we refer to \Cref{subsec:smoothness} for the definition of $\texttt{N}_\phi$, and invoking \cite[Theorem 4.2(3)]{patie:2016} allows us to conclude that, for any $q \geq 0$ and $a > 0$,
\begin{equation*}
|W_\phi(a + ib)| \stackrel{\infty}{=} \bigO(|b|^{-q})
\end{equation*}
uniformly on bounded $a$-intervals, so that for any $q \geq 0$ and $a > -\frac{1}{t}$
\begin{equation*}
|\M_{\sigma_t}(a+ib)| \stackrel{\infty}{=} \bigO(|b|^{-q})
\end{equation*}
uniformly on bounded $a$-intervals. By Mellin inversion we get $\sigma(dx) = \sigma_t(x)dx$ for each $t > 0$ and, from similar arguments as given in the proof of \Cref{thm:smoothness-t}\ref{item-2:thm:smoothness-t}, we get the Mellin-Barnes representation
\begin{equation*}
\sigma_t(x) = \frac{1}{2\pi i } \int_{c-i\infty}^{c+i\infty} x^{-z} W_\phi(tz-t+1) dz,
\end{equation*}
valid for any $c > 1 - \frac{1}{t}$. The change of variables $z \mapsto \frac{(z-1)}{t} + 1$ reveals that
\begin{equation}
\label{eq:Mellin-Barnes-sigma}
\sigma_t(x) = \frac{1}{2 \pi i t} \int_{c-i\infty}^{c+i\infty} x^{-\frac{(z-1)}{t} - 1} W_\phi(z) dz,
\end{equation}
for any $c > 0$, and using \Cref{thm:smoothness-t}\ref{item-2:thm:smoothness-t} to identify the right-hand side of \eqref{eq:Mellin-Barnes-sigma} we establish, for all $t > 0$, the equality
\begin{equation*}
\sigma_t(x) = \frac{1}{t} x^{\frac{1-t}{t}} \nu_1(x^\frac{1}{t})
\end{equation*}
 where $\nu_1(dx) = \nu_1(x)dx$. This identity allows us to use the asymptotic behavior of $\nu_1$ described in \cite[Theorem 5.5]{patie:2015} to get that
\begin{equation*}
\sigma_t(x) \stackrel{\infty}{\sim} \frac{C_\phi}{t\sqrt{2\pi}} x^{\frac{1-t}{t}} \sqrt{\varphi'(x^\frac{1}{t}}) \exp\left(-\int_\sf{k}^{x^\frac{1}{t}} \frac{\varphi(r)}{r} dr\right)
\end{equation*}
where $C_\phi > 0$ is a constant depending on $\phi$ and $\varphi:[\sf{k},\infty) \to [0,\infty)$ is the continuous inverse of $\phi$. Repeating, mutatis mutandis, the arguments from \Cref{thm:threshold-result}\ref{item-3:thm:threshold-result} we conclude that $X^t$ is moment indeterminate for $t > \frac{2}{\delta}$, and the last claim is straightforward.

\subsubsection{Proof of \Cref{thm:Lin-conjecture}\ref{item-4:thm:Lin-conjecture}}

The proof is the same as the one of \Cref{thm:threshold-result}\ref{item-4:thm:threshold-result} after observing that the assumptions imply the factorization of moment sequences
\begin{equation*}
W_\phi(tn+1) = W_{\frac{\phi}{\vartheta}}(tn+1) W_\vartheta(tn+1)
\end{equation*}
valid for any $t > 0$ and $n \geq 0$.

\subsubsection{Proof of \Cref{thm:Lin-conjecture}\ref{item-2:thm:Lin-conjecture}}

When $\phi \in \B_d$ \Cref{prop:reference-function} guarantees, for any $\alpha \in (0,1)$ and suitable $\mathfrak{m}$, that $\frac{\phi}{\phi_{\alpha,\mathfrak{m}}} \in \B$. Applying \Cref{thm:Lin-conjecture}\ref{item-4:thm:Lin-conjecture} it follows that $X^t(\phi)$ is indeterminate for any $t$ such that $X^t(\phi_{\alpha,\mathfrak{m}})$ is indeterminate. However, $\phi_{\alpha,\mathfrak{m}}(u) = (u+\mathfrak{m})^\alpha$, so by a combination of \Cref{prop:reference-function} and some straightforward asymptotic analysis one gets that $\phi_{\alpha,\mathfrak{m}} \in \B_{\asymp}$ with $\beta_{\phi_{\alpha,\mathfrak{m}}} = \alpha > 0$ and $\limsup_{u \to \infty} u^{-\alpha} \phi_{\alpha,\mathfrak{m}}(u) < \infty$. From \Cref{prop:reference-function}\ref{item-1:prop:reference-function} we get that $\phi_{\alpha,\mathfrak{m}}$ is a complete Bernstein function and hence, in particular, $\phi_{\alpha,\mathfrak{m}} \in \B_\J$. Thus \Cref{thm:Lin-conjecture}\ref{item-3:thm:Lin-conjecture} gives that $X^t(\phi_{\alpha,\mathfrak{m}})$ is moment indeterminate if and only if $t > \frac{2}{\alpha}$, from which we conclude that $X^t(\phi)$ is indeterminate for $t > \frac{2}{\alpha}$. Since $\alpha \in (0,1)$ is arbitrary we get that $X^t(\phi)$ is moment indeterminate if $t > 2$, and from \Cref{thm:Lin-conjecture}\ref{item-1:thm:Lin-conjecture} we get moment determinacy for $t \leq 2$, which finishes the proof.

\bibliographystyle{abbrv}

\begin{thebibliography}{10}

\bibitem{akhiezer:1965}
N.~I. Akhiezer.
\newblock {\em The classical moment problem and some related questions in
  analysis}.
\newblock Translated by N. Kemmer. Hafner Publishing Co., New York, 1965.

\bibitem{balkema:1993}
A.~A. Balkema, C.~Kl{\"u}ppelberg, and S.~I. Resnick.
\newblock Densities with {G}aussian tails.
\newblock {\em Proc. London Math. Soc. (3)}, 66(3):568--588, 1993.

\bibitem{balkema:1995}
A.~A. Balkema, C.~Kl{\"u}ppelberg, and U.~Stadtm{\"u}ller.
\newblock Tauberian results for densities with {G}aussian tails.
\newblock {\em J. London Math. Soc. (2)}, 51(2):383--400, 1995.

\bibitem{berg:2005}
C.~Berg.
\newblock On powers of {S}tieltjes moment sequences. {I}.
\newblock {\em J. Theoret. Probab.}, 18(4):871--889, 2005.

\bibitem{berg:2007}
C.~Berg.
\newblock On powers of {S}tieltjes moment sequences. {II}.
\newblock {\em J. Comput. Appl. Math.}, 199(1):23--38, 2007.

\bibitem{berg:2018}
C.~Berg.
\newblock A two-parameter extension of the {U}rbanik semigroup.
\newblock arXiv:1802.00993v1 [math.CV], 02 2018.

\bibitem{berg:1984}
C.~Berg, J.~P.~R. Christensen, and P.~Ressel.
\newblock {\em Harmonic analysis on semigroups}, volume 100 of {\em Graduate
  Texts in Mathematics}.
\newblock Springer-Verlag, New York, 1984.
\newblock Theory of positive definite and related functions.

\bibitem{berg:2004}
C.~Berg and A.~J. Dur{\'a}n.
\newblock A transformation from {H}ausdorff to {S}tieltjes moment sequences.
\newblock {\em Ark. Mat.}, 42(2):239--257, 2004.

\bibitem{berg:2015}
C.~Berg and J.~L. L{\'o}pez.
\newblock Asymptotic behaviour of the {U}rbanik semigroup.
\newblock {\em J. Approx. Theory}, 195:109--121, 2015.

\bibitem{bingham:1989}
N.~H. Bingham, C.~M. Goldie, and J.~L. Teugels.
\newblock {\em Regular variation}, volume~27 of {\em Encyclopedia of
  Mathematics and its Applications}.
\newblock Cambridge University Press, Cambridge, 1989.

\bibitem{blumenthal:1961}
R.~M. Blumenthal and R.~K. Getoor.
\newblock Sample functions of stochastic processes with stationary independent
  increments.
\newblock {\em J. Math. Mech.}, 10:493--516, 1961.

\bibitem{deng:2015}
C.-S. Deng and R.~L. Schilling.
\newblock On shift {H}arnack inequalities for subordinate semigroups and moment
  estimates for {L}{\'e}vy processes.
\newblock {\em Stochastic Process. Appl.}, 125(10):3851--3878, 2015.

\bibitem{folland:1999}
G.~B. Folland.
\newblock {\em Real analysis}.
\newblock Pure and Applied Mathematics (New York). John Wiley \& Sons, Inc.,
  New York, second edition, 1999.
\newblock Modern techniques and their applications, A Wiley-Interscience
  Publication.

\bibitem{fourati:2011}
S.~Fourati and W.~Jedidi.
\newblock Some remarks on the class of {B}ernstein functions and some
  subclasses.
\newblock hal-00649018, 2011.

\bibitem{hirsch:2013}
F.~Hirsch and M.~Yor.
\newblock On the {M}ellin transforms of the perpetuity and the remainder
  variables associated to a subordinator.
\newblock {\em Bernoulli}, 19(4):1350--1377, 2013.

\bibitem{janson:2010}
S.~Janson.
\newblock Moments of gamma type and the {B}rownian supremum process area.
\newblock {\em Probab. Surv.}, 7:1--52, 2010.

\bibitem{jurek:1985}
Z.~J. Jurek.
\newblock Relations between the {$s$}-self-decomposable and self-decomposable
  measures.
\newblock {\em Ann. Probab.}, 13(2):592--608, 1985.

\bibitem{jurek:2009}
Z.~J. Jurek.
\newblock On relations between {U}rbanik and {M}ehler semigroups.
\newblock {\em Probab. Math. Statist.}, 29(2):297--308, 2009.

\bibitem{jurek:1993}
Z.~J. Jurek and J.~D. Mason.
\newblock {\em Operator-limit distributions in probability theory}.
\newblock Wiley Series in Probability and Mathematical Statistics: Probability
  and Mathematical Statistics. John Wiley \& Sons, Inc., New York, 1993.
\newblock A Wiley-Interscience Publication.

\bibitem{kuznetsov:2013}
A.~Kuznetsov and J.~C. Pardo.
\newblock Fluctuations of stable processes and exponential functionals of
  hypergeometric {L}{\'e}vy processes.
\newblock {\em Acta Appl. Math.}, 123:113--139, 2013.

\bibitem{lawrie:1994}
J.~B. Lawrie and A.~C. King.
\newblock Exact solution to a class of functional difference equations with
  application to a moving contact line flow.
\newblock {\em European J. Appl. Math.}, 5(2):141--157, 1994.

\bibitem{letemplier:2015}
J.~Letemplier and T.~Simon.
\newblock On the law of homogeneous stable functionals.
\newblock arXiv:1510.07441v2 [math.PR], 10 2015.

\bibitem{lin:2017}
G.~D. Lin.
\newblock On powers of the {C}atalan number sequence.
\newblock arXiv:1711.01536v1 [math.CO], 11 2017.

\bibitem{patie:2015}
P.~Patie and M.~Savov.
\newblock Spectral expansions of non-self-adjoint generalized {L}aguerre
  semigroups.
\newblock {\em Mem. Amer. Math. Soc.}, 179 pp., 2019.

\bibitem{patie:2016}
P.~Patie and M.~Savov.
\newblock Bernstein-{G}amma functions and exponential functionals of {L}\'{e}vy
  processes.
\newblock {\em Electron. J. Probab.}, 23(75):1--101, 2018.

\bibitem{patie:2017}
P.~Patie and M.~Savov.
\newblock Cauchy problem of the non-self-adjoint {G}auss-{L}aguerre semigroups
  and uniform bounds for generalized {L}aguerre polynomials.
\newblock {\em J. Spectr. Theory}, 7(3):797--846, 2017.

\bibitem{patie:2018a}
P.~Patie and A.~Vaidyanathan.
\newblock Non-classical {T}auberian and {A}belian type criteria for the moment
  problem.
\newblock  arXiv:1804.10721v1 [math.PR], 19pp., 2018.

\bibitem{sato:2013}
K.-i. Sato.
\newblock {\em L{\'e}vy processes and infinitely divisible distributions},
  volume~68 of {\em Cambridge Studies in Advanced Mathematics}.
\newblock Cambridge University Press, Cambridge, 2013.
\newblock Translated from the 1990 Japanese original, Revised edition of the
  1999 English translation.

\bibitem{schilling:2012}
R.~L. Schilling, R.~Song, and Z.~Vondra\v{c}ek.
\newblock {\em Bernstein functions}, volume~37 of {\em De Gruyter Studies in
  Mathematics}.
\newblock Walter de Gruyter \& Co., Berlin, second edition, 2012.
\newblock Theory and applications.

\bibitem{schmudgen:2017}
K.~Schm\"{u}dgen.
\newblock {\em The moment problem}, volume 277 of {\em Graduate Texts in
  Mathematics}.
\newblock Springer, Cham, 2017.

\bibitem{shohat:1943}
J.~A. Shohat and J.~D. Tamarkin.
\newblock {\em The {P}roblem of {M}oments}.
\newblock American Mathematical Society Mathematical surveys, vol. I. American
  Mathematical Society, New York, 1943.

\bibitem{stieltjes:1894}
T.-J. Stieltjes.
\newblock Recherches sur les fractions continues.
\newblock {\em Ann. Fac. Sci. Toulouse Sci. Math. Sci. Phys.}, 8(4):J1--J122,
  1894.

\bibitem{titchmarsh:1986}
E.~C. Titchmarsh.
\newblock {\em Introduction to the theory of {F}ourier integrals}.
\newblock Chelsea Publishing Co., New York, third edition, 1986.

\bibitem{tyan:1975}
S.-G. Tyan.
\newblock {\em The structure of bivariate distribution functions and their
  relation to Markov processes}.
\newblock PhD thesis, Princeton University, 1975.

\bibitem{urbanik:1992}
K.~Urbanik.
\newblock Functionals on transient stochastic processes with independent
  increments.
\newblock {\em Studia Math.}, 103(3):299--315, 1992.

\bibitem{urbanik:1995}
K.~Urbanik.
\newblock Infinite divisibility of some functionals on stochastic processes.
\newblock {\em Probab. Math. Statist.}, 15:493--513, 1995.
\newblock Dedicated to the memory of Jerzy Neyman.

\bibitem{webster:1997}
R.~Webster.
\newblock Log-convex solutions to the functional equation {$f(x+1)=g(x)f(x)$}:
  {$\Gamma$}-type functions.
\newblock {\em J. Math. Anal. Appl.}, 209(2):605--623, 1997.

\end{thebibliography}

\end{document}